\documentclass[11pt,reqno]{amsart}
\usepackage{latexsym,amsmath}
\usepackage[left=3cm,top=2.5cm,right=3cm,bottom=2.5cm]{geometry}
\usepackage{color}
\usepackage{amsthm}
\usepackage{amssymb}

\usepackage{enumerate}
\usepackage{tikz}
\DeclareMathOperator{\ex}{ex}

\usepackage{datetime}

\usepackage[utf8]{inputenc}
\usepackage[T1]{fontenc}
\usetikzlibrary{patterns}

\newtheorem{theorem}{Theorem}[section]
\newtheorem{cor}[theorem]{Corollary}
\newtheorem{lemma}[theorem]{Lemma}

\newtheorem{proposition}[theorem]{Proposition}
\newtheorem{claim}[theorem]{Claim}
\newtheorem{conjecture}[theorem]{Conjecture}
\newtheorem{definition}[theorem]{Definition}
\theoremstyle{definition}

\newtheorem{remark}[theorem]{Remark}
\newcommand\eps{\varepsilon}

\newcommand\norm[1]{\left\Vert #1 \right\Vert}
\newcommand\normp[2]{\left\Vert #2 \right\Vert_{#1}}

\begin{document}

\allowdisplaybreaks[2]

\newcommand{\carlosa}[1]{{\bf [~Carlos 30/04/16:\ } {\em #1}{\bf~]}}

\pagestyle{plain}
\thispagestyle{empty}
\footskip=30pt

\title[Edge-colorings of graphs avoiding complete graphs with a prescribed coloring]
  {Edge-colorings of graphs avoiding complete graphs with a prescribed coloring}

\author[F. S. Benevides]{Fabr\'{i}cio Siqueira Benevides}
\address{Departamento de Matemática, Centro de Ci\^encias, UFC --
  Campus do Pici, Bloco 914, 60451--760 Fortaleza, CE, Brazil} \email{\tt fabricio@mat.ufc.br}

\author[C. Hoppen]{Carlos Hoppen}
\address{Instituto de Matem\'{a}tica e Estat\'{i}stica, Universidade
 Federal do Rio Grande do Sul, Avenida Bento Gon\c{c}alves,
 9500, 91509-900, Porto Alegre, Brazil}
\email{choppen@ufrgs.br}

\author[R. M. Sampaio]{Rudini Menezes Sampaio}
\address{Departamento de Computa\c c\~ao, Centro de Ci\^encias, UFC --
  Campus do Pici, Bloco 910, 60451--760 Fortaleza, CE, Brazil} \email{\tt rudini@lia.ufc.br}

\thanks{The first and the thrid author were partially supported by FUNCAP and CNPq. The second author was partially supported by FAPERGS~(Proc.\,2233-2551/14-0) and CNPq (Proc.~448754/2014-2 and~308539/2015-0).}

\begin{abstract}
Given a graph $F$ and an integer $r \ge 2$, a partition $\widehat{F}$ of the edge set of $F$ into at most $r$ classes, and a graph $G$, define $c_{r, \widehat{F}}(G)$ as the number of $r$-colorings of the edges of $G$ that do not contain a copy of $F$ such that the edge partition induced by the coloring is isomorphic to the one of $F$. We think of $\widehat{F}$ as the pattern of coloring that should be avoided. The main question is, for a large enough $n$, to find the (extremal) graph $G$ on $n$ vertices which maximizes $c_{r, \widehat{F}}(G)$. This problem generalizes a question of Erd{\H o}s and Rothschild, who originally asked about the number of colorings not containing a monochromatic clique (which is equivalent to the case where $F$ is a clique and the partition $\widehat{F}$ contains a single class). We use H\"{o}lder's Inequality together with Zykov's Symmetrization to prove that, for any $r \geq 2$, $k \geq 3$ and any pattern $\widehat{K_k}$ of the clique $K_k$, there exists a complete multipartite graph that is extremal. Furthermore, if the pattern $\widehat{K_k}$ has at least two classes, with the possible exception of two very small patterns (on three or four vertices), every extremal graph must be a complete multipartite graph. In the case that $r=3$ and $\widehat{F}$ is a rainbow triangle (that is,  where $F=K_3$ and each part is a singleton), we show that an extremal graph must be an almost complete graph.  Still for $r=3$, we extend a result about monochromatic patterns of Alon, Balogh, Keevash and Sudakov to some patterns that use two of the three colors, finding the exact extremal graph. For the later two results, we use the Regularity and Stability Method.
\end{abstract}

\maketitle


\section{Introduction}

For any fixed graph $F$, we say that a graph $G$ is $F$-free if it does not contain $F$ as a subgraph. Finding the maximum number of edges among all $F$-free $n$-vertex graphs, and determining the class of $n$-vertex graphs that achieve this number is known as the \emph{Tur\'{a}n problem} associated with $F$, which was solved for complete graphs in~\cite{turan}. The maximum number of edges in an $F$-free $n$-vertex graph is denoted by $\ex(n,F)$ and the $n$-vertex graphs that achieve this bound are called \emph{$F$-extremal}. Tur\'{a}n has found the value of $\ex(n, F)$ for the case where $F$ is a clique $K_k$ on $k$ vertices, for any $k \geq 3$. Moreover, he showed that the $K_k$-free graph on $n$ vertices which has $\ex(n, K_k)$ edges is unique (up to isomorphism). This graph is a complete multipartite graph with $k-1$ parts of sizes as equal as possible, and we will denote it by $T_{k-1}(n)$. This problem and its many variants have been widely studied and there is a vast literature related with it. For more information, see Keevash~\cite{Kee11} and the references therein.

In connection with a question of Erd\H{o}s and Rothschild~\cite{Erd74}, several authors have investigated the following related problem. Instead of looking for $F$-free $n$-vertex graphs, they were interested in \emph{edge-colorings} of graphs on $n$ vertices such that \emph{every color class is $F$-free}. (We observe that edge colorings in this work are not necessarily proper colorings.) More precisely, given an integer $r \geq 1$ and a graph $F$ containing at least one edge, one considers the function $c_{r,F}(G)$ that associates, with the graph $G$, the number of $r$-colorings of the edge set of $G$ for which there is no monochromatic copy of $F$. Similarly as before, the problem consists of finding $c_{r,F}(n)$, the maximum of $c_{r,F}(G)$ over all $n$-vertex graphs $G$.

The function $c_{r,F}(n)$ has been studied for several classes of graphs, such as complete graphs~\cite{ABKS,PY12,yuster}, odd cycles~\cite{ABKS}, matchings~\cite{forbm}, paths and stars~\cite{paths}. The hypergraph analogue of this problem has also been considered, see for instance~\cite{kneser,independent,LP11,LPS11}, and there has been recent progress in the context of additive combinatorics~\cite{HJ16}. There is a straightforward connection between $c_{r,F} (n)$ and $\ex(n,F)$, namely
\begin{eqnarray} \label{ser1}
c_{r, F}(n) \geq r^{\ex(n,F)}  \textrm{ for every }n\geq 2,
\end{eqnarray}
as any $r$-coloring of the edges of an $F$-extremal $n$-vertex graph is trivially $F$-free, and there are precisely $r^{\ex(n,F)}$ such colorings. For $r \in \{2,3\}$ the inequality~\eqref{ser1} is actually an equation for several graph classes, such as complete graphs~\cite{ABKS,yuster}, odd cycles~\cite{ABKS} and matchings~\cite{forbm}. On the other hand, for $r \geq 4$ and all connected $F$, one may easily show that $c_{r,F}(n) > r^{\ex(n,F)}$ (see~\cite{ABKS} for non-bipartite graphs and~\cite[Proposition~3.4]{paths} for bipartite graphs).

Here we consider a natural generalization of the above, which was first studied by Lefmann and one of the current authors~\cite{HL15}. Given a $k$-vertex graph $\widehat{F}$ colored with at most $r$ colors, we consider the number of $r$-edge-colorings of a larger graph $G$ that avoids the \emph{color pattern} of $\widehat{F}$. Here, a pattern is defined as any partition of the edge set (of a graph $F$), and the pattern given by a coloring is simply the pattern induced by the color classes. Notice that in a pattern we ignore the name of the colors. We let $c_{r,\widehat{F}}(G)$ denote the number or $r$-colorings of $G$ which contain no $k$-vertex subgraph whose color pattern is isomorphic to the one of $\widehat{F}$. We say that a coloring that avoids the pattern of $\widehat{F}$ is $\widehat{F}$-free. When the context is clear we omit the subscripts in $c_{r,\widehat{F}}(G)$ and also refer to an $\widehat{F}_k$-free $r$-coloring simply as a \emph{good coloring}. Also, a graph $G$ on $n$ vertices is called $(r,\widehat{F})$-extremal (or simply extremal), when $c_{r,\hat{F}}(G) = c_{r,\hat{F}}(n)$.

We note that Balogh~\cite{balogh06} had also considered a multicolored variant of the original Erd\H{o}s-Rothschild problem. Given $\widehat{F}$ and $G$ as before, he considered the number $C_{r, \widehat{F}}(G)$ of $r$-colorings of $G$ which do not contain a copy of $F$ colored \emph{exactly} as $\widehat{F}$ (that is, in his version, we were not allowed to permute the colors). Observe that $c_{r,\widehat{F}}(G) \le C_{r, \widehat{F}}(G)$, but the notions of these two quantities are different. For example, consider the case where $\widehat{F}$ is a coloring of $F$ that uses only one of the $r$ colors, say ``blue''. In this case, $c_{r,\widehat{F}}(G)$ counts the number of colorings of $G$ that avoids monochromatic copies of $F$, agreeing with the previous definition of $c_{r,F}(G)$, while $C_{r, \widehat{F}}(G)$ is the number of colorings of $G$ which does not contain a blue copy $\widehat{F}$ (but may contain copies of $F$ in other colors). As another example, if one considers $r$-colorings of $G$, but the coloring of $\widehat{F}$ uses at most $r-1$ of the colors, then the complete graph $K_n$ is always extremal for $C_{r,\widehat{F}}(n)$, as the missing color may be used for any edge and hence may be used to extend colorings of any $n$-vertex graph $G$ to colorings of $K_n$. However, colorings may not always be extended in this way in the case where we want to avoid color patterns, that is, when we are searching for the extremal graphs of $c_{r,\widehat{F}}(n)$.

Balogh~\cite{balogh06} proved that in the case where $r=2$ and $\widehat{F}$ is a $2$-coloring of a clique that uses both colors then $C_{2,\widehat{F}}(n) = 2^{\ex(n,F)}$ for $n$ large enough, so the Tur\'{a}n graph $T_{k-1}(n)$ allows the maximum number of $2$-colorings with no colored copy of $F$.  However, the picture changes if we consider $3$-colorings with no rainbow triangles (pattern $R_0$ in Figure~\ref{fig:somepatters}): Balogh also observed that, if we color the complete graph $K_n$ with any two of the three colors available, there is no rainbow copy of $K_3$, which gives at least $3\cdot 2^{\binom{n}{2}} - 3 \gg 3^{\ex (n, K_3)}=3^{n^2/4 + o(n^2)}$ distinct colorings avoiding rainbow triangles. (As usual, we say that two positive functions $g,f$ satisfy $g(n) \ll f(n)$ if $\lim_{n \rightarrow \infty} g(n)/f(n)=0$.) 

 
In this paper, we focus on the case where $r \ge 3$ and the pattern is given by any edge-coloring of a \emph{clique} that is not monochromatic. The paper has two parts which use very different techniques. In the first part, corresponding to Section~\ref{sec:generalresults}, we shall use some ideas from the so called Zykov's symmetrization \cite{zykov1949} (which also yields one of the classical proofs of Tur\'{a}n's theorem), together with H\"{o}lder's Inequality for a certain vector space, to prove a general result that works for arbitrary patterns (including the monochromatic one). First we show the following:

\def\theoremthereIsExtremalMultipartite{Let $\widehat{F}_k$ be any $r$-coloring of $K_k$. For every natural $n$, there exists a complete multipartite graph on $n$ vertices which is $(r, \widehat{F}_k)$-extremal.}

\begin{theorem}\label{theorem:thereIsExtremalMultipartite}
\theoremthereIsExtremalMultipartite
\end{theorem}

Very recently, Pikhurko, Staden and Yilma \cite{PSY16} have obtained a similar result, albeit for a different extension of the original problem about monochromatic patterns (their forbidden patterns are still only monochromatic cliques, but they forbid cliques of different sizes for different colors).

In addition, we also proved that whenever the pattern is non-monochromatic and is different than two  particular small patterns, then \emph{every} extremal graph is a complete multipartite one.
\def\theoremExtremalstatement{Let $r \ge 2$ and $k \ge 3$ be given and let $\widehat{F}_k$ be a $r$-coloring of $K_k$ which is not monochromatic and is different from the pattern $T_0$. Also assume that if $r=2$ then $\widehat{F}_k$ is different from the pattern $P_2$ (see Figure~\ref{fig:somepatters}). Then \emph{every} $(r,\widehat{F}_k)$-extremal graph is a complete multipartite graph.}

\begin{theorem} \label{thm:nonmonoExtremalsAreMultipartite}
	\theoremExtremalstatement
\end{theorem}

\begin{figure}[thb]
\begin{center}
	\begin{tikzpicture}[scale=1.6]
	\pgfsetlinewidth{1pt}

	\tikzset{vertex/.style={circle, fill=black!50, draw}}

	\begin{scope}[xshift=-4cm, yshift=0]
		\node [vertex] (a) at (0,0) {};
		\node [vertex] (b) at (0,1) {};
		\node [vertex] (c) at (1,1) {};

		\draw (a) -- (b) -- (c);
		\draw[dashed, thin] (a) -- (c);

		\draw[black] (0.5,-0.2) node [below] {$T_0$};
	 \end{scope}

	\begin{scope}[xshift=-2cm, yshift=0]
		\node [vertex] (a) at (0,0) {};
		\node [vertex] (b) at (0,1) {};
		\node [vertex] (c) at (1,1) {};

		\draw[dotted, line width=1mm] (a) -- (b);
		\draw (b) -- (c);
		\draw[dashed, thin] (a) -- (c);

		\draw[black] (0.5,-0.2) node [below] {$R_0$};
	 \end{scope}

	\begin{scope}[xshift=0, yshift=0]
		\node [vertex] (a) at (0,0) {};
		\node [vertex] (b) at (0,1) {};
		\node [vertex] (c) at (1,1) {};
		\node [vertex] (d) at (1,0) {};

		\draw (b) -- (c) -- (d) -- (b);
		\draw (a) -- (b);
		\draw (a) -- (c);
		\draw[dashed, thin] (a) -- (d);

		\draw[black] (0.5,-0.2) node [below] {$P_1$};
	 \end{scope}
	 
	 \begin{scope}[xshift=2cm, yshift=0]
		\node [vertex] (a) at (0,0) {};
		\node [vertex] (b) at (0,1) {};
		\node [vertex] (c) at (1,1) {};
		\node [vertex] (d) at (1,0) {};
		\draw (a) -- (c);
		\draw [dashed, thin] (a) -- (b);
		\draw (b) -- (c) -- (d) -- (b);

		\draw[dashed, thin] (a) -- (d);
		\draw[black] (0.5,-0.2) node [below] {$P_2$};
	 \end{scope}

	 \begin{scope}[xshift=4cm, yshift=0]
		\node [vertex] (a) at (0,0) {};
		\node [vertex] (b) at (0,1) {};
		\node [vertex] (c) at (1,1) {};
		\node [vertex] (d) at (1,0) {};

		\draw (b) -- (c) -- (d) -- (b);
		\draw[dotted, line width=1mm] (a) -- (b);
		\draw (a) -- (c);
		\draw[dashed, thin] (a) -- (d);
		\draw[black] (0.5,-0.2) node [below] {$P_3$};
	 \end{scope}
	\end{tikzpicture}
\caption{Some special patterns of colorings: $T_0$, $P_1$, $P_2$, use two colors, and $R_0$ and $P_3$ use three colors.}
\label{fig:somepatters}
\end{center}
\end{figure}
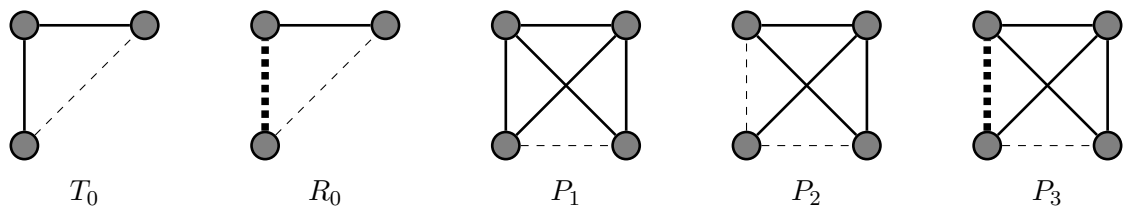

We remark that when $r=2$, the previously mentioned results~\cite{ABKS,balogh06,yuster} for $C_{2,\widehat{F}_k}(n)$ already imply that $c_{2,\widehat{F}_k}(n)=2^{\ex(n,K_k)}$ for every $k\ge 3$ and every $2$-coloring $\widehat{F}_k$ of the complete graph $K_k$. In particular, in the case where $\widehat{F}_k = P_2$ and $r=2$, our proof of Theorem~\ref{thm:nonmonoExtremalsAreMultipartite} does not work, but we already know the exact optimum. Furthermore, if  $r=3$ and $\widehat{F}_3$ is the pattern $T_0$ in Figure~\ref{fig:somepatters}, then our main theorem in Section~\ref{sec:twocolors} implies that the (only) extremal graph is the Tur\'{a}n graph. We believe that the conclusion in Theorem~\ref{thm:nonmonoExtremalsAreMultipartite} actually works for any pattern given by a coloring of a clique.

A strong implication of Theorem~\ref{thm:nonmonoExtremalsAreMultipartite} is that, in order to find any extremal graph of the pattern in the statement, we only have to find the number of vertices in each class that maximizes the number of colorings. We believe that if the pattern has some symmetry, then the number of vertices in each class must be the same. However, we have no indication that this must be true for all patterns. As a matter of fact, if we do not require the forbidden graph to be complete, there are instances where the extremal graph is complete multipartite, but the classes are not equitable, see~\cite{HL15,forbm} in the case of matchings, or where the extremal graph is not even complete multipartite, see~\cite{paths}. 

The exact extremal graph is known only for a very small values of $r$ and very particular patterns. In most cases, when we do know the exact extremal graph for $c_{r, \widehat{F}}(n)$, it happens that we have equality in \eqref{ser1} and the extremal is the graph on $n$ vertices and $\ex(n,F)$ edges. Pikhurko and Yilma \cite{PY12} have found the exact extremal graph for two cases where we do not have equality in \eqref{ser1}: when $r=4$ and $\widehat{F}$ is either a monochromatic $K_3$ or $K_4$. Of course, such extremal graphs for $K_t$, for both $t=3$ and $t=4$, are still complete multipartite graphs (in fact, they are also Tur\'{a}n graphs, but different from $T_{t-1}(n)$), as our result in Theorem~\ref{thm:nonmonoExtremalsAreMultipartite}.

In the second part of the paper, we focus on the case where $r=3$ and try to get more precise results for a specific family of patterns. We extend to multicolored patterns the method of Alon, Balogh, Keevash and Sudakov~\cite{ABKS} (see also~\cite{balogh06}), which uses Szemer\'{e}di's Regularity Lemma. Our original motivation was only to look at the pattern $T_0$ (the two-colored triangle) and $R_0$ (the rainbow triangle) in Figure~\ref{fig:somepatters}. We conjecture that the extremal graph for $R_0$ is the complete graph (when $r=3$). In Section~\ref{sec:rainbow}, our main result (Theorem~\ref{thm:approx_rainbow}) is an approximate version of this which says that an extremal is an almost complete graph (in two different ways). Finally, in Section~\ref{sec:twocolors}, we prove that for any pattern $\widehat{F}_k$ (generated by a coloring of $K_k$) that satisfies a certain stability condition the (only) extremal graph for $c_{r, \widehat{F}_k}(n)$ is the Tur\'{a}n graph $T_{k-1}(n)$, for each $n$ large enough. Afterwards, we show that such stability is satisfied by patterns that uses only two colors and one of which induces a graph of small Ramsey number. This includes the pattern $R_0$. Together, they add up to the following main theorem.

\def\theoremTwocolors{Let $k \geq 3$ and let $\widehat{F}$ be a pattern of $K_k$ with two classes, one of which induces a graph $J$ such that $R(J,J) \leq k$. Then for $n$ sufficiently large the number of $3$-edge colorings that avoid $\widehat{F}$ is maximized for the Tur\'{a}n graph $T_{k-1}(n)$.}

\begin{theorem} \label{thm_patterns}
	\theoremTwocolors
\end{theorem}

Recently, there has also been progress in finding graphs that admit the largest number of $r$-colorings avoiding some pattern of a complete graph for $r \geq 4$ colors. Typically, the results obtained focus on rainbow patterns of $K_k$, that is, patterns where all edges are assigned different colors, and show that the Tur\'{a}n graph $T_{k-1}(n)$ is optimal for large $n$ as long as $r \geq r_0(k)$. For instance, in the case $k=3$, this is known for $r_0=5$ (see~\cite{HLO15}). In~\cite{HLO15}, the authors also extend the general method of~\cite{ABKS} to multicolored patterns, but the results in the first part allow us to shorten it slightly.


\section{Results that hold for every coloring pattern $\widehat{F}_k$ of a clique $K_k$.}
\label{sec:generalresults}

For this section, $r \ge 2$ and $k \ge 3$ are natural numbers and $\widehat{F}_k$ is any $r$-coloring of a complete graph $K_k$.

For a vector $\vec{x}$ indexed by a set $T$, we will denote by $x(t)$ the value of $x$ at coordinate $t$, where $t\in T$. We will use $\normp{p}{\vec{x}}$ to denote the $\ell_p$-norm of $\vec{x}$, so for $p\in (0, \infty)$ we have 
$$\normp{p}{x} = \left(\sum_{t\in T} |x(t)|^p\right)^{1/p}.$$
Moreover, for a sequence of vectors $x_1, \ldots, x_s$, each indexed by $T$, we will denote their pointwise product by $\prod_{k=1}^{n} \vec{x}_k$, that is, the vector $y$ such that for each $t\in T$ we have $y(t) = \prod_{k=1}^{n}x_k(t)$.

\begin{definition} \label{def:profilevector} If $H$ is a subgraph of a graph $G$ and $\widehat{H}$ is an $\widehat{F}_k$-free $r$-coloring  of $H$, we denote by $c_{r,\widehat{F}_k}(G \mid \widehat{H})$ the number of ways to $r$-color the edges in $E(G) - E(H)$ in such a way that the resulting coloring is still $\widehat{F}_k$-free. For a single vertex $v \in V(G)-V(H)$, we use the notation $c_{r,\widehat{F}_k}(v,\widehat{H})$ for the number of ways to $r$-color the edges from $v$ to $V(H)$ (again avoiding $\widehat{F}_k$). We also define $\vec{v}_H$ as the vector indexed by the $\widehat{F}_k$-free $r$-colorings of $H$, whose coordinate corresponding to a coloring $\widehat{H}$ is given by $\vec{v}_H(\widehat{H}) = c(v,\widehat{H})$.
\end{definition}

We have the following immediate proposition.

\begin{proposition}\label{prop:indepset}
If $H$ is an induced subgraph of $G$ such that $S = V(G)-V(H)$ is an independent set in $G$, and $\widehat{H}$ is an $(r, {F_k})$-free coloring of $H$, then
\[
 	c(G\mid \widehat{H}) = \prod_{v\in S} c(v,\widehat{H}).
 \] 
\end{proposition} 
\begin{proof}
It follows trivially from the fact that there is no $K_k$ that contains two vertices of $S$ and therefore the choice of colors of the edges incident to a vertex of $S$ does not affect the colors of edges incident to other vertices of $S$.
\end{proof}

We will need the inequality below, known as the Generalized H\"{o}lder's Inequality (stated here for the particular case of the counting measure on a finite set). For a more general version see the book \cite{wheeden2015measure} (chapter 8, exercise 6).

\begin{lemma}[H\"{o}lder's Inequality] \label{lemma:holderfull}
Assume that $r \in (0, \infty)$ and $p_1, p_2, \ldots, p_s \in (0, \infty]$ are such that $$\sum_{k=1}^{n} \frac{1}{p_k} = \frac{1}{r},$$ and let $\vec{x_1}, \ldots, \vec{x_s}$ be complex-valued vectors indexed by a common set $T$. We have
$$\normp{r}{\prod_{k=1}^{s} \vec{x}_k} \le \prod_{k=1}^{s}\normp{p_k}{\vec{x}_k}.$$
Furthermore, equality happens above if and only if for every $i, j \in [s]$ and every $t\in T$ we have $$\frac{|\vec{x}_i(t)|^{p_i}}{\normp{p_i}{\vec{x}_i}}= \frac{|\vec{x}_j(t)|^{p_j}}{\normp{p_j}{\vec{x}_j}}.$$
\end{lemma}

We will actually use it only in the following particular form.

\begin{cor}\label{cor:holder}
Let $\vec{x}_1, \ldots, \vec{x}_s$ be complex-valued vectors indexed by the same set. We have
$$\normp{1}{\prod_{k=1}^{s} \vec{x}_k} \le \prod_{k=1}^{s}\normp{s}{\vec{x}_k}.$$
\end{cor}
\begin{proof}
Take $r=1$ and, for $1 \le i \le s$, take $p_i = s$ in the statement of Lemma~\ref{lemma:holderfull}.
\end{proof}

 \begin{remark}
 When $s=2$, the inequality in Corollary \ref{cor:holder}, is equivalent to the Cauchy-Schwartz inequality: 
$\left< \vec{x}_1, \vec{x}_2 \right> \le \norm{\vec{x}_1}\norm{\vec{x}_2}$.
 \end{remark}

\begin{definition}
We say that two vertices are \emph{twins} if they are non-adjacent and have the same neighborhood. \emph{Cloning} a vertex $v$ of a graph $G$ means to create a new graph $\widetilde{G}$ whose vertex set is $V(G)\cup\{\widetilde{v}\}$ where $\widetilde{v}$ is a new vertex which is a twin of $v$.
\end{definition}
 
For the next lemma we consider the following operation: take an independent set $S$ of a graph $G$, select a particular vertex $v\in S$, delete all vertices in $S-v$ and add $s-1$ new twins of $v$. The result is a new graph which has at least as many good colorings as $G$.

\begin{lemma}\label{lemma:replaceIndSet}
Let $\widehat{F}_k$ be any $r$-coloring of $K_k$. Let $G$ be a graph on $n$ vertices, $S \subset V(G)$ be an independent set with $s = |S|$, $H = G-S$, and $A=V(G)-S$. There exists a vertex $v \in S$ with the following property: if we construct the graph $\widetilde{G}$ with $V(\widetilde{G}) = V(H) \cup \widetilde{S}$, where $\widetilde{S}$ is an independent set on $s$ vertices, each of which is a twin of $v$, and $\widetilde{G}[A] = G[A]$, then $c_{r,\widehat{F}_k}(\widetilde{G})\ge c_{r,\widehat{F}_k}(G)$.
\end{lemma}
\begin{proof}
Let $S$ be any independent set in $G$, and let $H = G-S$. For each $u \in S$, consider the vector $\vec{u}_H$ as in Definition~\ref{def:profilevector}. By Proposition~\ref{prop:indepset}, the total number of $\widehat{F}_k$-free $r$-coloring of $G$ is
$$c(G) = \sum_{\widehat{H}} c(G \mid \widehat{H}_k) = \sum_{\widehat{H}} \prod_{u\in S} c(u, \widehat{H}) = \normp{1}{\prod_{u\in S} \vec{u}_H},$$
where the sums are taken over all possible ${F_k}$-free $r$-colorings $\widehat{H}$ of $H$. (For the last equality we also used that every coordinate of $\vec{u}_H$ is non-negative).

Let $v$ be a vertex in $S$ for which $\normp{s}{\vec{v}_H}$ is maximum. This fact, together with Hölder's Inequality (Corollary~\ref{cor:holder}), gives us:
\begin{equation} \label{eq:usingholder}
\normp{1}{\prod_{u\in S} \vec{u}_H} \le \prod_{\vec{u}\in S}\normp{s}{\vec{u}_H} \le \normp{s}{\vec{v}_H}^s.
\end{equation}

On the other hand, for the graph $\widetilde{G}$ defined in the statement of this lemma, we have:

$$c(\widetilde{G}) = \sum_{\widehat{H}} c(v, \widehat{H})^s = \normp{s}{\vec{v}_H}^s.$$
 
Therefore, $c(\widetilde{G}) \ge c(G)$.
\end{proof}

\begin{cor}\label{cor:ColorVectorsAreEqual} If $G$ is an $(r, \widehat{F}_k)$-extremal graph, $S\subseteq V(G)$ is an independent set and $H = G-S$, then for every $u, v \in S$ we must have $\vec{v}_H = \vec{u}_H$.
\end{cor}
\begin{proof}
Let $G$ be a graph as in the statement. Let $\widetilde{G}$ be the graph defined in Lemma~\ref{lemma:replaceIndSet}. Since $G$ is extremal, and $c_{r,\widehat{F}_k}(\widetilde{G}) \ge c_{r,\widehat{F}_k}(G)$, we must have $c_{r,\widehat{F}_k}(\widetilde{G}) = c_{r,\widehat{F}_k}(G)$. Therefore, we must also have equality in both inequalities in \eqref{eq:usingholder}. From the second one, it follows that for every $u, v \in S$, we must have $\normp{s}{\vec{u}_H} = \normp{s}{\vec{v}_H}$. From the first one, where we used Lemma \ref{lemma:holderfull}, the fact that $\normp{s}{\vec{u}_H} = \normp{s}{\vec{v}_H}$, together with the fact that all our vectors have only positive entries, implies that $\vec{v}_H = \vec{u}_H$.
\end{proof}

\begin{cor}\label{corolary:cloning}
If $G$ is an $(r, \widehat{F}_k)$-extremal graph, and $u, v\in V(G)$ are any non-adjacent vertices, then deleting $v$ and cloning $u$ produces a graph that is also extremal.
\end{cor}
\begin{proof}
 Since $G$ is extremal, by Corollary \ref{cor:ColorVectorsAreEqual} with $S = \{u, v\}$ and $G_{uv} = G-\{u,v\}$, we must have $\vec{u}_{G_{uv}} = \vec{v}_{G_{uv}}$, therefore replacing $v$ by a twin of $u$ (or $u$ by a twin of $v$) does not change the number of colorings of the graph.
\end{proof}

By repeatedly applying Corollary~\ref{corolary:cloning} above, we can easily show that \emph{there exists} a complete multipartite graph on $n$ vertices which is $(r, \widehat{F}_k)$-extremal. Although this is a direct consequence of Corollary~\ref{corolary:cloning}, we spell out the details. On the other hand, showing that (for non-monochromatic patterns) \emph{every} extremal is a complete multipartite graph will require more work.

\begingroup
\def\thetheorem{\ref{theorem:thereIsExtremalMultipartite}}
\begin{theorem}
\theoremthereIsExtremalMultipartite
\end{theorem}
\addtocounter{theorem}{-1}
\endgroup

\begin{proof}[Proof of Theorem~\ref{theorem:thereIsExtremalMultipartite}]
Let $G$ be any $(r, \widehat{F}_k)$-extremal graph on $n$ vertices. We will build a sequence of extremal graphs, each on $n$ vertices, say $G_0, G_1, \ldots, G_t$, where $G_0 = G$, and $G_t$ is a complete multipartite graph. We do it in such a way that, for $i\ge 1$, we have $V(G_i) = S_1 \cup \cdots \cup S_i \cup R_{i}$, where for every $j \in \{1, \ldots, i\}$, the set $S_j$ is an independent set and every vertex in $S_j$ is adjacent to every vertex outside $S_j$ (including those in $R_i$), but we have no control of the edges inside $R_i$. It will also hold that $R_t \subset R_{t-1} \subset \cdots \subset R_1 \subset V(G)$, and $R_t$ will be independent.

To simplify the notation, we also define $R_0 = V(G_0) = V(G)$. Assume that we have constructed $G_i$, for some $i\ge 0$. If $R_{i}$ is an independent set, we have found a complete multipartite graph which is extremal, so we can set $t = i$ and stop. Otherwise, let $v_i$ be any vertex of $R_i$ that has a neighbor in $R_i$. Note that, by the definition of $G_i$, all non-neighbors of $v_i$ belong to $R_i$. Let $\overline{d}_i$ be the number of non-neighbors of $v_i$. We can obtain $G_{i+1}$ applying Corollary \ref{corolary:cloning} successively $\overline{d}_{i}$ times, deleting each non-neighbor of $v_i$ and adding twins of $v_i$ (one by one). Let $S_{i+1}$ be the set formed by $v_i$ and its new twins and let $R_{i+1}$ to be the set of neighbors of $v_i$ in $R_i$. Observe that $R_{i+1}$ is strictly smaller than $R_i$ since it does not contain $v_i$. It is also important to notice that, at every step when we use Corollary \ref{corolary:cloning} we apply it to the whole graph $G_i$ and not only to $G_i[R_i]$.
\end{proof}

Observe that in the proof of Theorem~\ref{theorem:thereIsExtremalMultipartite} we may select $v_i$ as the vertex with the largest degree in $R_i$. By doing this, we obtain, starting from an extremal graph $G$, a complete multipartite graph that has at least as many edges as $G$. In the next lemma, we show that if $G$ is not complete multipartite itself,  we can find another complete multipartite extremal graph by only deleting edges of $G$.

\begin{lemma}[Edge deletion lemma] \label{lemma:containsmultipartite}
Let $\widehat{F}_k$ be any $r$-coloring of graph $K_k$ and $r \ge 2$ be a natural number. Let $G$ be an $(r, \widehat{F}_k)$-extremal graph and assume that $G$ is not a complete multipartite graph. For any $u, v, w$  such that $uv, uw \notin E(G)$ and $vw \in E(G)$, if we delete the edge $vw$, then the resulting graph is still extremal.
\end{lemma}

\begin{proof}
Let $G$ be a graph as in the statement. Since $G$ is not a complete multipartite graph, there exists vertices $u, v, w$ such that $uv, uw \notin E(G)$ and $vw \in E(G)$. Let $u,v,w$ be any such vertices.

Let $H = G - \{u, v, w\}$, and $H^x = G[V(H)\cup x]$ for $x \in \{u,v,w\}$. Let $G'$ be the graph obtained from $G$ by deleting the edge $vw$ (but not the vertices $u$ or $v$), and let $G^*$ be the graph obtained from $H^u$ by adding another two clones of $u$, say $u_1$ and $u_2$. By Corollary \ref{corolary:cloning}, the graph $G^*$ is also extremal, as we may first apply the replacement operation to the pair $u,v$ (deleting $v$ and adding $u_1$) and apply it again to the pair $u,w$. Therefore, $c(G) = c(G^*)$.

Applying Proposition \ref{prop:indepset} to $G^*$ with $S = \{u,u_1,u_2\}$, we have
$$c(G^*) = \sum_{\widehat{H}} c(G^* \mid \widehat{H}) = \sum_{\widehat{H}} c(u,\widehat{H})^3 = \normp{3}{\vec{u}_H}^3,$$
where the sum is taken over all $\widehat{F}_k$-free $r$-colorings of $H$.

Observe that, with an analogous computation, if we start from $H$ and add three clones of $w$ instead of $u$, the resulting graph has $\normp{3}{\vec{w}_H}^3$ good colorings. But we do not know if such graph is extremal, so we have only
\begin{equation} \label{eq:normUvsW}
\normp{3}{\vec{w}_H}^3 \le \normp{3}{\vec{u}_H}^3.
\end{equation}

On the other hand, since there are no edges from $u$ to $\{v, w\}$, we can compute $c(G)$ as follows:
\begin{equation}
\begin{split}
c(G) &= \sum_{\widehat{H}} \left(c(u,\widehat{H})\cdot c(G - u \mid \widehat{H})\right) \\
	 &= \sum_{\widehat{H}}\left( c(u,\widehat{H})\cdot  \left( \sum_{\widehat{H^w} \mid \widehat{H}} c(v, \widehat{H^w}) \right)\right). \label{eq:usingHolderA}
\end{split}
\end{equation}

Here, the inner sum is taken over the good colorings of $H^w$ that extend a given good coloring of $H$, that is, over the colorings of the edges from $w$ to $H$, for which the resulting coloring is good. By Corollary \ref{cor:ColorVectorsAreEqual}, since $G$ is extremal and $uv \notin E(G)$, we have $\vec{v}_{H^{w}} = \vec{u}_{H^{w}}$, that is $c(v,\widehat{H^w}) = c(u,\widehat{H^w})$ for every $\widehat{H^w}$. Finally, note that $c(u,\widehat{H^w})$ does not depend on the colors of the edges from $w$ to $H$, so $c(u,\widehat{H^w}) = c(u,\widehat{H})$. Therefore,

\begin{align}	 
c(G) &= \sum_{\widehat{H}}\left( c(u,\widehat{H}) \left( \sum_{\widehat{H^w} \mid \widehat{H}} c(u, \widehat{H}) \right)\right) \\
     &= \sum_{\widehat{H}}\left( c(u,\widehat{H}) c(u,\widehat{H}) \sum_{\widehat{H^w} \mid \widehat{H}} 1 \right) \\
	 &= \sum_{\widehat{H}} c(u,\widehat{H})^2c(w,\widehat{H}) \\
	 &\le \normp{3}{\vec{u}_H} \normp{3}{\vec{u}_H} \normp{3}{\vec{w}_H} \label{eq:usingHolderB}\\
	 &\le \normp{3}{\vec{u}_H}^3. \label{eq:usingHolderC}
\end{align}

Notice that to get \eqref{eq:usingHolderB} we used Hölder's Inequality (Corollary~\ref{cor:holder}), and \eqref{eq:usingHolderC} follows from \eqref{eq:normUvsW}. Finally, since $c(G) = \normp{3}{\vec{u}_H}^3$, we must have equality in both \eqref{eq:usingHolderB} and \eqref{eq:usingHolderC}, which in turn leads to $\normp{3}{\vec{u}_H} = \normp{3}{\vec{w}_H}$. The equality condition in Lemma~\ref{lemma:holderfull} implies that $\vec{u}_H = \vec{w}_H$. Analogously, $\vec{u}_H = \vec{v}_H$. It follows that $$c(G^*) = \sum_{\widehat{H}} c(u,\widehat{H})c(v,\widehat{H})c(w,\widehat{H}) = c(G').$$
\end{proof}

Finally, we use Lemma~\ref{lemma:containsmultipartite} to prove our main result of this section, which we restate below.

\begingroup
\def\thetheorem{\ref{thm:nonmonoExtremalsAreMultipartite}}
\begin{theorem}
\theoremExtremalstatement
\end{theorem}
\addtocounter{theorem}{-1}
\endgroup

\begin{proof}
	Let $\widehat{F}_k$ be a $r$-coloring as in the statement. Suppose that there exists an $(r,\widehat{F}_k)$-extremal graph $G$ which is not a complete multipartite graph. Let $u, v, w, H$, $H^v$, and $H^w$ be defined as in the proof of Lemma~\ref{lemma:containsmultipartite}. At the end of the proof, we concluded $\vec{u}_H = \vec{w}_H  = \vec{v}_H$, so for every coloring $\widehat{H}$ of $H$ we have $c(u,\widehat{H}) = c(w,\widehat{H}) = c(v,\widehat{H})$. We also noticed that, for every extension of $\widehat{H}$ to a coloring $\widehat{H^w}$, we have $c(u,\widehat{H^w}) = c(u,\widehat{H})$.

    Now note that, since $u$ and $v$ are not adjacent, by Corollary~\ref{cor:ColorVectorsAreEqual}, we have $\vec{u}_{H^w} = \vec{v}_{H^w}$, that is, $c(u, \widehat{H^w}) = c(v, \widehat{H^w})$ for every $\widehat{H^w}$. From the previous equalitities, if follows that, for every $\widehat{F}_k$-free extension $\widehat{H^w}$ of $\widehat{H}$, we must have 
   	\begin{equation} \label{eq:extension}
   	c(v,\widehat{H^w}) = c(v,\widehat{H}).
   	\end{equation}

    Our goal here is to get a contradiction from this fact (which implies that such $G$ cannot exist). We only need to find a $r$-coloring of $H$ and an extension of it to $H^w$, which is $\widehat{F}_k$-free and such that equation \eqref{eq:extension} does not hold. We will split the proof into cases, depending on the pattern of $\widehat{F}_k$. In each case we proceed as follows. We fix a particular good coloring $\widehat{H^w}$ of $H^w$ and consider the coloring $\widehat{H}$ induced by it in $H$. Let $\mathcal{H}(v)$ and $\mathcal{H}^w(v)$ denote the set of extensions of $\widehat{H}$ to $H^v$ and of $\widehat{H^w}$ to $G-u$, respectively. To find a contradiction to~\eqref{eq:extension}, we show that there is an injective mapping $\phi \colon \mathcal{H}(v) \rightarrow \mathcal{H}^w(v)$ that is not surjective. 


	We say that a coloring of $\widehat{F}_k$ is \emph{almost monochromatic} if it is not monochromatic and there exists a vertex $x \in F_k$ such that all edges not incident to $x$ have the same color, say color 1, and there is at least one edge incident to $x$ that is also of color 1. We call such $x$ the \emph{special} vertex. 

The remainder of the proof splits the analysis into four cases. Figure~\ref{fig:cases} illustrates how colorings are extended in each case.

\noindent {\bf Case 1}:
$\widehat{F}_k$ is \emph{not} almost monochromatic. Let $\widehat{H^w}$ be the coloring that assigns color blue to all edges of $H^w$, so that $\widehat{H}$ is a blue coloring of $H$. To define the injective mapping  $\phi \colon \mathcal{H}(v) \rightarrow \mathcal{H}^w(v)$, for any extension of  $\widehat{H}$ to $H^v$, consider the same extension of $\widehat{H^w}$ to the edges between $v$ and $H$ and assign blue to the edge $vw$. By definition of good coloring, there is no $\widehat{F}_k$ in the extension to $H^v$ or in $H^w$, so any copy of $\widehat{F}_k$ must be induced by a set that contains $vw$. However, any such set, induces a coloring that is almost monochromatic (in which $v$ plays the role of the special vertex $x$). On the other hand, consider the coloring of $G-u$ where all edges are blue, with the exception of the edge $vw$, which is colored red.  Any pattern contained in this coloring is either monochromatic or almost monochromatic, and therefore is different from $\widehat{F}_k$. However, it is not in the image of $\phi$.
	
\noindent {\bf Case 2}: $\widehat{F}_k$ \emph{is} almost monochromatic and is different from the patterns $T_0, P_1, P_2, P_3$ of Figure~\ref{fig:somepatters}.  Let $\widehat{H^w}$ be such that all edges inside $H$ are blue and the ones from $w$ to $H$ are red. To define $\phi$, for any good coloring that extends $\widehat{H}$ to the edges between $v$ and $H$, extend it by coloring $vw$ with red. As before, we only need to check that any pattern that contains the edge $vw$ is not equal to $\widehat{F}_k$. Notice that here we must have $k\ge 4$ (as $\widehat{F}_k$ is almost monochromatic and different from $T_0$). Suppose that there is an almost monochromatic pattern that contains $vw$. Note that it must contain exactly two vertices of $H$, one of which is the image of the special vertex $x$. Because all edges in $F_k-x$ have the same color, $\widehat{F}_k$ must be equal to $P_1, P_2$ or $P_3$, a contradiction. To see that $\phi$ is not surjective, let all edges from $v$ to $H$ be red and the edge $vw$ be blue. It is easy to check that the only pattern which is almost monochromatic and is contained in this coloring is $T_0$. 

\noindent {\bf Case 3}: $\widehat{F}_k$ \emph{is} $P_1$ or $P_3$, given in Figure~\ref{fig:somepatters}. Let $\widehat{H^w}$ be such that all edges inside $H$ are blue and the ones from $w$ to $H$ are red. To define $\phi$, given a good coloring that extends $\widehat{H}$ to the edges between $v$ and $H$, extend it to $G-v$ by coloring $vw$ with blue. It is easy to see that this cannot produce $P_1$ or $P_3$ using $vw$. Again, this function $\phi$ is not surjective, as we may color all edges between $v$ and $H^w$ with blue and let $vw$ be red.
		
\noindent {\bf Case 4}: $\widehat{F}_k$ \emph{is} $P_2$, given in Figure~\ref{fig:somepatters}. In this case we assume $r\ge3$.  Let $\widehat{H^w}$ be such that all edges inside $H$ are blue and the ones from $w$ to $H$ are red. To define $\phi$, given a good coloring that extends $\widehat{H}$ to the edges between $v$ and $H$, extend it to $G-v$ by coloring $vw$ with a third color, say green. Clearly, any four vertices that contains $v$ and $w$ induce a pattern that uses at least three colors, and thus is not equal to $P_2$. Note that the extension of $\widehat{H^w}$ such that all edges from $v$ to $H^w$ are red does not contain the pattern $P_2$, so that $\phi$ is not surjective.

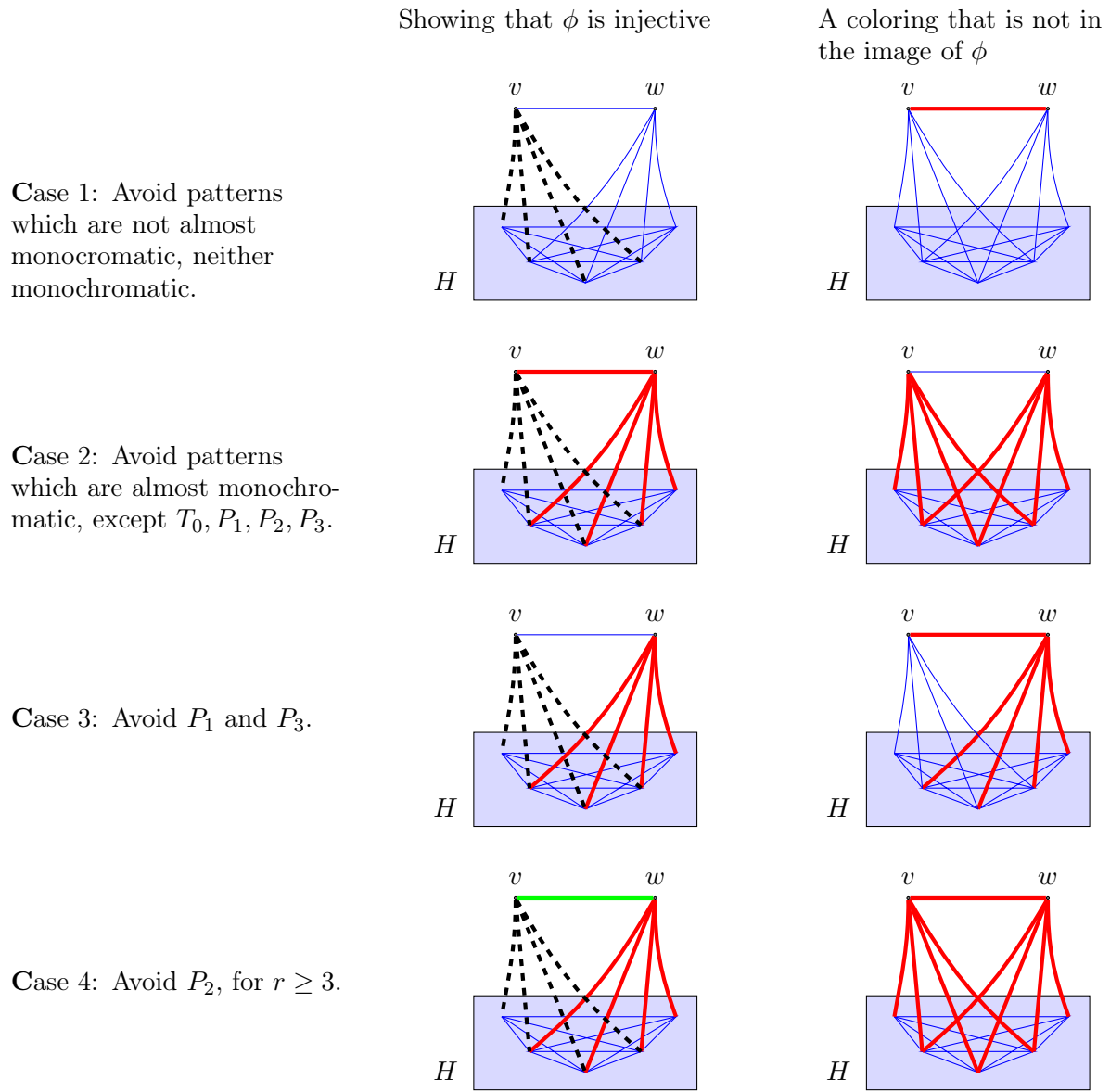
\begin{figure}
\begin{center}
\newcommand{\coloring}[4]{%
	\coordinate (A) at (-1.6, -0.25);
	\coordinate (B) at ( 1.6, 1.1);
	\draw (-2,0) node {$H$};
	\coordinate (p0) at (-1.2, 0.8);
	\coordinate (p1) at (-0.8, 0.3);
	\coordinate (p2) at (0, 0);
	\coordinate (p3) at (0.8, 0.3);
	\coordinate (p4) at (1.3, 0.8);

	\node[label=$v$, vertex] (v) at (-1,2.5){};
    \node[label=$w$, vertex] (w) at ( 1,2.5){};

	\draw[fill=#1!15] (A) rectangle (B) ;
	\draw[#1] (p0) -- (p1) -- (p2) -- (p3) -- (p4) -- (p0);
	\draw[#1] (p0) -- (p2) -- (p4) -- (p1) -- (p3) -- (p0);

    \draw[#2] (w) to [out=240, in=40] (p1);
    \draw[#2] (w) -- (p2);
    \draw[#2] (w) -- (p3);
	\draw[#2] (w) to [out=270, in=110] (p4);
    
    \draw[#3] (v) to [out=270, in=80] (p0);
    \draw[#3] (v) -- (p1);
    \draw[#3] (v) -- (p2);
    \draw[#3] (v) to [out=300, in=140] (p3);

    \draw[#4] (v) -- (w);
}

\def\colorHbox{blue}

\begin{tikzpicture}[text depth=0.25ex]
    \tikzset{vertex/.style={circle, fill=black!50, draw, inner sep=0pt}}

	\matrix[column sep=0.4cm,row sep=0.5cm]
	{
	    &
	    \node[text width=14em] (Extending) {Showing that $\phi$ is injective};&
	    \node[text width=12em] (Particular) {A coloring that is not in \newline the image of $\phi$};
	    \\
	    \node[text width=4.9cm, anchor=south] (Case1) at (0,1) {{\textbf Case 1:} Avoid patterns which are not almost monocromatic, neither monochromatic.};&
	    \coloring{\colorHbox}{blue}{dashed, ultra thick}{blue} &
	    \coloring{\colorHbox}{blue}{blue}{red, line width=1.6pt} &
	    \\
	    \node[text width=4.9cm, anchor=south] (Case2) at (0,1) {{\textbf Case 2:} Avoid patterns \newline which are almost monochromatic, except $T_0, P_1, P_2, P_3$.}; &
	    \coloring{\colorHbox}{red, line width=1.6pt}{dashed, ultra thick}{red, line width=1.6pt} &
	    \coloring{\colorHbox}{red, line width=1.6pt}{red, line width=1.6pt}{blue} &
	    \\
	    \node[text width=4.9cm, anchor=south] (Case3) at (0,1) {{\textbf Case 3:} Avoid $P_1$ and $P_3$.};&
	    \coloring{\colorHbox}{red, line width=1.6pt}{dashed, ultra thick}{blue} &
	    \coloring{\colorHbox}{red, line width=1.6pt}{blue}{red, line width=1.6pt} &
	    \\
	   	\node[text width=4.9cm, anchor=south] (Case3) at (0,1) {{\textbf Case 4:} Avoid $P_2$, for \mbox{$r\ge3$}.};&
	    \coloring{\colorHbox}{red, line width=1.6pt}{dashed, ultra thick}{green, ultra thick} &
	    \coloring{\colorHbox}{red, line width=1.6pt}{red, line width=1.6pt}{red, line width=1.6pt} &
	    \\
	};
\end{tikzpicture}
\end{center}
\caption{How to color and to extend a coloring in each case.}
\label{fig:cases}
\end{figure}
\end{proof}

\begin{remark}
Note that if $\widehat{F}_3$ is a rainbow coloring of $K_3$, then it is treated in Case~1 
 of Theorem~\ref{thm:nonmonoExtremalsAreMultipartite}. The proof that we gave here does not work for monochromatic pattern simply because our colorings of $H^w$ always contain monochromatic cliques.
\end{remark}


\section{The case of $3$-colorings - Auxiliary results}

In the remainder of this paper, we shall only be concerned with colorings with three colors. In Sections~\ref{sec:rainbow} and ~\ref{sec:twocolors}, our proofs will be based on the Regularity Method of Szemer\'{e}di together with some stability results. Here, we give the necessary definitions and state the main results that we shall use.

Given two disjoint non-empty sets of vertices $X$ and $Y$ of a graph $G$, we let $E(X,Y)$ denote the set of edges with one end in $X$ and the other one in $Y$. We also set $e(X,Y)=|E(X,Y)|$ and let $d(X, Y)=\frac{e(X,Y)}{|X||Y|}$ denote the edge density between $X$ and $Y$.
  
\begin{definition} \label{szem:76}
  Let $G=(V,E)$ be a graph and let $0 < \eps \le 1$. We say that a pair 
  $(A,B)$ of two disjoint subsets of $V$ is \emph{$\eps$-regular} (with 
  respect to $G$) if 
    $$
      |d(A',B')-d(A,B)| < \eps
    $$
  holds for any two subsets $A'\subset A$, $B'\subset B$ with
  $|A'| > \eps |A|$, $|B'| > \eps |B|$. 
 \end{definition}

\begin{definition}
Given a graph $G = (V,E)$, a partition $V = V_1\cup \ldots \cup V_t$ is called \emph{$\eps$-regular} (with respect to $G$) if: 
\begin{itemize}
	\item[(a)] $\lvert |V_i| - |V_j| \rvert \le 1$ for every $i, j \in \{1, \ldots, t\}$, and 
	\item[(b)] $(V_i, V_j)$ is $\eps$-regular for all but at most $\eps t^2$ of the pairs $(V_i, V_j)$ where $i\neq j$.
\end{itemize}
\end{definition}

In our proofs, we shall make use of a colored version of the Szemer\'{e}di Regularity Lemma~\cite{Sze78} stated in~\cite{KS}.
\begin{lemma}\label{colored_regularity}
For every $m,\eps>0$ and integer $r$, there exist $n_0$ and $M$ such that, if the edges of a graph $G$ of order $n\geq n_0$ are $r$-colored, say $E(G)=E_1\cup \cdots \cup E_r$, then there is a partition of the vertex set $V(G)=V_1\cup \cdots \cup V_t$ with $m\leq t\leq M$ which is $\eps$-regular simultaneously with respect to all graphs $G_i=(V,E_i)$ for $i=1,\ldots,r.$
\end{lemma}

A partition as in Lemma \ref{colored_regularity} will be called a \emph{multicolored $\eps$-regular partition}. Given such a partition and given a color $\sigma \in [r]$, we can define a \emph{cluster graph} associated with color $\sigma$ as follows. Given $\eta > 0$, the graph $R_\sigma = R_\sigma(\eta)$ is defined on the vertex set $[t]$ so that $\{i,j\} \in E(R_\sigma)$ if and only if $(V_i,V_j)$ is an $\eps$-regular pair with edge density at least $\eta$ with respect to the subgraph of $G$ induced by the edges of color $\sigma$.

 We may also define the \emph{multicolored cluster graph} $R$ associated with this partition: the vertex set is $[t]$ and $e = \{i,j\}$ is an edge of $R$ if $e \in E(R_\sigma)$ for some $\sigma \in [r]$. Each edge $e$ in $H$ is assigned the list of colors $L_{e} = \left\{\sigma \in [r] \, | \; e \in E(R_\sigma)\right\}$. Given a colored graph $\widehat{F}$, we say that a multicolored cluster graph $R$ contains $\widehat{F}$ if $R$ contains a copy of $F$ such that the color of each edge (with respect to $\widehat{F}$) is contained in the list of the corresponding edge in $R$. More generally, if $F$ is a graph with color pattern $P$, we say that $R$ contains $(F,P)$ if it contains some colored copy of $F$ with pattern $P$.

One of the main advantages of considering cluster graphs are embedding results that ensure that some substructure found within a cluster graph can also be found in the original graph. In the present work, the following embedding result will be particularly useful. It is stated in terms of $3$-colorings because of our setting, but the same statement would hold for $r$ colors. The proof is quite standard and follows the arguments in the proof of~\cite[Theorem~2.1]{KS}.
\begin{lemma} \label{lemma_embedding} For every $\eta > 0$ and every positive integer $k$, there exist $\varepsilon = \varepsilon (\eta, k) > 0$ and a positive integer $n_0(\eta, k)$ with the following property. Suppose that $G$ is a $3$-edge colored graph on $n > n_0$ vertices with a multicolored $\varepsilon$-regular partition $V = V_1 \cup \cdots \cup V_t$ which defines the multicolored cluster graph $R = R(\eta)$. Let $F$ be a fixed $k$-vertex graph with a prescribed color pattern $\widehat{F}$. If $R$ contains $\widehat{F}$, then the graph $G$ also contains $\widehat{F}$.	
\end{lemma}

The following classical stability result will also be used in our proofs.
\begin{theorem}\cite{erdos_stability,simonovits_stability}\label{simonovits_stability}
For every $\alpha>0$ there exist $\beta>0$ and $n_0$ such that any $K_k$-free graph on $n \geq n_0$ vertices with at least $\ex(n,K_k)-\beta n^2$ edges has a partition $V=V_1 \cup \cdots \cup V_{k-1}$ of the vertex set with $\sum e(V_i)<\alpha n^2$.
\end{theorem}

We will also need the entropy function, which we will denote by $H(x)$, and is defined as $H(x)=-x\log_2(x)-(1-x)\log_2(1-x)$, for $0<x<1$. It will be useful for the well known estimate
\begin{eqnarray}\label{entropy}
\binom{a}{xa} \leq 2^{H(x)a}.
\end{eqnarray}
Note that $\lim_{x \rightarrow 0^+} H(x)=0$.


\section{$3$-colorings avoiding a rainbow triangles}
\label{sec:rainbow}

Throughout this section we let $\widehat{F}_3$ be a $3$-colored rainbow $K_3$, that is, one in which all edges have different colors. Here, we will use the Regularity Method to show that every $(3, \widehat{F}_3)$-extremal graph is an `almost complete' graph. Recall that we already know that it is a complete multipartite graph.

For reasons that will be clear later, we will need to solve the problem of maximizing the value $w(G)$ defined below, over the set of graphs with a given number of vertices.

\begin{definition}
\label{def:w}
Given a graph $G$, let $w:E(G)\to\{2, 3\}$ be the function that gives weight $2$ or $3$ to the edges of $G$ in such a way that every edge that belongs to some triangle gets weight 2 and all the remaining edges get weight $3$. Define $w(G)$ to be the product of the weight of the edges of $G$. 
\end{definition}

The following lemma tells us that for a given number of vertices, the value of $w(G)$ is maximum when $G$ is a complete graph.

\begin{lemma} \label{lemma:GraphMaxW} Given a graph $G$ on $t$ vertices, the function $w(G)$ defined above satisfies $w(G) \leq 2^{\binom{t}{2}}$.
\end{lemma}

Our proof of Lemma~\ref{lemma:GraphMaxW} (which works for all values of $t$) is based, again, on Zykov's Symmetrization. The next lemma is a more general result, but works only for large values of $t$. Since, in this article, we will only need results for large values of $t$, we postpone the proof of Lemma~\ref{lemma:GraphMaxW} to the appendix.

\begin{lemma}\label{lemma:stabilityforw} Let $G$ be a graph on $t > 1000$ vertices. Attribute weights to the edges of $G$ as in Definition \ref{def:w}. For $i \in \{2, 3\}$, let $e_i$ be the number of edges of weight $i$ and let $\bar{e}$ be the number of edges in the complement $\bar{G}$ of $G$. If $\bar{e} \le t^2/4$, then $w(G) = 2^{e_2}3^{e_3} \le 2^{\binom{t}{2}}2^{-0.16\bar{e}}$. Furthermore, if $\bar{e} > t^2/4$ then $w(G) < 3^{t^2/4} \ll 2^{\binom{t}{2}}$. 
\end{lemma}
\begin{proof}

Let $G$ be a graph such that $\bar{e} \le t^2/4$.
We will double count the number of pairs $(uv, ab)$ where $uv$ is an edge of weight 3 of $G$ and $ab$ is an edge in the complement of $G$ and $\{u,v\}\cap\{a,b\} \neq \emptyset$. Let $T$ be the number of such pairs.

For every edge $uv$ of weight 3, we must have $N(u)\cap N(v) = \emptyset$. Therefore, $d(u)+d(v) \le t$. This implies that $\bar{d}(u)+\bar{d}(v) \ge 2(t-1)-t = t-2$. Therefore, there are at least $t-2$ edges $ab$ of $\bar{G}$ which are incident with $uv$. This implies that $T \ge e_3(t-2)$.

Now, for each non-edge $ab$, we want to bound the number of edges of weight 3 which are incident with $a$ or $b$. That is, we want an upper bound on $|N_3(a)|+|N_3(b)|$ (noting that we are counting edges and not the number of vertices in $N_3(a)\cup N_3(b)$). We claim that for every $a \in V(G)$, we have that $|N_3(a)| \le t/\sqrt{2}+1$. In fact, since the edges $ua$ where $u \in N_3(a)$ have weight 3, they do no belong to any triangle and therefore $N_3(a)$ is an independent set. This implies that $$\frac{(|N_3(a)|-1)^2}{2} \le \binom{|N_3(a)|}{2} \le \bar{e} \le \frac{t^2}{4}.$$ Therefore, $|N_3(a)| \le t/\sqrt{2}+1$ as desired. The same bound holds for $|N_3(b)|$. It follows that $|N_3(a)|+|N_3(b)| \le t\sqrt{2} + 2$. This implies that $T \le (t\sqrt{2} + 2)\bar{e}$.

Comparing the upper bound and the lower bound for $T$, we have that: $(t\sqrt{2} + 2)\bar{e} \ge e_3(t-2)$, which implies $e_3 \le (\sqrt{2}+\frac{2\sqrt{2}+2}{t-2})\bar{e} < (\sqrt{2}+0.01)\bar{e}$. We conclude that
$$2^{e_2}3^{e_3} = 2^{e_2+e_3+\bar{e}} \left(\frac{3}{2}\right)^{e_3}\frac{1}{2^{\bar{e}}} = 2^{\binom{t}{2}} 2^{\log_2(3/2)e_3-\bar{e}} \le 2^{\binom{t}{2}} 2^{(\log_2(3/2)(\sqrt{2}+0.01)-1)\bar{e}} \le 2^{\binom{t}{2}} 2^{(0.84-1)\bar{e}}.$$

Finally, we note that the case where $\bar{e} > t^2/4$ is trivial.
\end{proof}

The result below establishes two approximate results about $(3,\widehat{F}_3)$-extremal graphs.

\begin{theorem}\label{thm:approx_rainbow}
The following hold for the rainbow triangle $\widehat{F}_3$.
\begin{itemize}
\item[(a)] For all $\delta>0$ there exists $n_0$ such that, if $G$ is a graph of order $n>n_0$, then $c_{3,\widehat{F}_3}(G) \le 2^{(1+\delta)n^2/2}$.
\item[(b)] For all $\xi>0$, there exists $n_1$ such that, if $G$ is a graph of order $n>n_1$ and $c_{3,\widehat{F}_3}(G) \geq 2^{\binom{n}{2}}$, then $|E(G)| \geq \binom{n}{2}-\xi n^2$.
\end{itemize}
\end{theorem}

\begin{proof}

For part (a), fix $\delta>0$ and consider $\eta>0$ such that $2\eta+H(\eta) \leq \delta/3$.
For this value of $\eta$, set $n_0'$ and $\eps$ given by Lemma~\ref{lemma_embedding}, where we further assume that $\eps<\eta/4$. Let $n_0''$ and $M$ be given by Lemma~\ref{colored_regularity} with $m=1/\eps$. 

Let $n_0>\max\{n_0',n_0''\}$, where additionally the inequality~\eqref{tapsi0} holds for all $n \geq n_0$. Consider a graph $G=(V,E)$ with $n \geq n_0$ vertices. We want to bound the number of $3$-colorings of $G$ that do not have a rainbow triangle. 

Fix an arbitrary 3-edge-coloring of $G$ with no rainbow triangle, and let $V_1 \cup \cdots \cup V_t$ be an $\eps$-regular partition given by Lemma~\ref{colored_regularity} associated with this coloring. Let $R_1$, $R_2$ and $R_3$ be the cluster graphs (with density $\eta>0$) associated with each color on the vertex set $\{1,\ldots,t\}$, and let $R$ be the corresponding multicolored cluster graph. 

First we bound the number of $3$-edge colorings of $G$ that could give rise to this particular partition and these cluster graphs. The number of edges that lie within some class of the partition is bounded above by $t \binom{n/t}{2} \leq n^2/(2t) \leq \eps n^2/2<\eta n^2/8$, while the number of edges joining a pair of vertices in classes that are not regular with respect to some color is at most $3 \eps t^2 \binom{n/t}{2} < 3 \eta n^2/8$. There are also at most $3\eta/4 \cdot \binom{n}{2} \leq 3 \eta n^2/8$ edges that join a pair of classes in which their color has density smaller than $\eta/4$. This adds to at most $\eta n^2$ edges. There are at most $\binom{n^2}{\eta n^2}$ ways to choose this set of edges and they can be colored in at most $3^{\eta n^2}$ different ways.

For any pair $(i,j)$ with $i < j$, the remaining edges joining $V_i$ to $V_j$ may be colored in at most $s_{i,j}$ ways, where $s_{i,j}$ is the number of cluster graphs amongst $R_1$, $R_2$ and $R_3$ for which $\{i,j\}$ is an edge. Since $e(V_i,V_j) \leq \left(n/t\right)^2$, there are at most $s_{i,j}^{n^2/t^2}$ ways to color these edges. Let $E_s$ be the set of edges that appear in exactly $s$ of the cluster graphs and denote $e_s=|E_s|$. 

This discussion implies that the number of potential 3-edge colorings of $G$ that could give rise to this vertex partition and these cluster graphs is at most
\begin{eqnarray} \label{eq:numberOfClusterGraphs}
&&\binom{n^2}{\eta n^2} 3^{\eta n^2} (1^{e_1}2^{e_2}3^{e_3})^{n^2/t^2}.
\end{eqnarray}

Notice that, the above estimate works for any coloring pattern that we want to avoid, not only for the rainbow triangle. So, we shall use it again in the proof of Lemma~\ref{lemma_patterns}, which is about a different pattern.

The term $3^{\eta n^2}$ may be replaced by the upper bound $2^{2\eta n^2}$, while the quantity $\binom{n^2}{\eta n^2} 3^{\eta n^2} $ may be bounded above by $2^{(H(\eta)+2\eta)n^2}$ because of~\eqref{entropy}.

Next, for an upper bound on $1^{e_1}2^{e_2}3^{e_3}$, note that this value may be obtained from $R$ by giving weight $i$ to the edges in $E_i$ and multiplying the weights of the edges. Clearly, there cannot be a triangle formed by three edges of $E_3$, otherwise we could find a rainbow triangle in the multicolored cluster graph, which, by Lemma~\ref{lemma_embedding}, would lead to a rainbow triangle in the original coloring. Therefore, $1^{e_1}2^{e_2}3^{e_3} \le w(R)$, where $w$ is defined as in Definition~\ref{def:w}, from which we derive  $1^{e_1}2^{e_2}3^{e_3} \le 2^{\binom{t}{2}} \le 2^{t^2/2}$ by Lemma~\ref{lemma:GraphMaxW}. This implies that $(1^{e_1}2^{e_2}3^{e_3})^{n^2/t^2} \le 2^{n^2/2}$.

To conclude the proof of part (a), note that the total number of vertex partitions is bounded above by $M^n$, while, for a given partition, the number of distinct multicolored cluster graphs is at most $2^{3M^2/2}$. As a consequence, we have
\begin{eqnarray}\label{tapsi0}
c_{3,\widehat{F}_3}(G) &\leq& M^n \cdot 2^{3M^2/2} \cdot 2^{(2\eta+H(\eta))n^2} \cdot 2^{n^2/2} \nonumber\\
  &\leq& 2^{(1+\delta)n^2/2}
\end{eqnarray}
by our choice of $n$ and $\eta$.

To prove part (b), assume that $G$ is a graph with at most $(1-\xi)\binom{n}{2}$ edges and consider $\eta>0$ such that
$$164 \eta+62 H(\eta) < \xi.$$ 
The other constants are fixed in terms of $\eta$ as in part (a) and $n_0$ is sufficiently large so that~\eqref{tapsi1} and~\eqref{tapsi2} hold for all $n \geq n_0$..

We proceed as in part (a), that is, we obtain a multicolored cluster graph $R$ for each $3$-edge-coloring of $G$ with no rainbow triangle. Given such a graph $R$ on $t$ vertices (recall that $1/\eps \leq t \leq M$), we consider two cases, according to whether $\overline{e}(R)=\binom{t}{2}-e(R) > \left(30 \eta+10 H(\eta)\right)t^2$, or whether this is not the case. In the former case, Lemma~\ref{lemma:stabilityforw} implies that 
$$w(R) \leq 2^{\binom{t}{2}(1-0.16 \overline{e}(R))}<2^{\binom{t}{2}(1-6 \eta - 2 H(\eta))}.$$ 
As in (a), summing over all possible partitions and multicolored cluster graphs in this case, the number of good colorings of $G$ is at most 
\begin{eqnarray}\label{tapsi1}
&&M^n \cdot 2^{3M^2/2} \cdot 2^{(2\eta+H(\eta))n^2} \cdot 2^{(1-6 \eta - 2 H(\eta))n^2/2} \leq M^n \cdot 2^{3M^2/2} \cdot 2^{\binom{n}{2}-\eta n^2} \leq \frac{1}{4} \cdot 2^{\binom{n}{2}}.
\end{eqnarray}

Next consider colorings such that $\overline{e}(R) \leq \left(40 \eta+15 H(\eta)\right)t^2$. By the proof of Lemma~\ref{lemma:stabilityforw}, we have $e_3 \leq 2 \overline{e} \leq \left(80 \eta+30 H(\eta)\right)t^2$. Once again, summing over possible partitions and cluster graphs, we obtain the following upper bound on the number of feasible $3$-edge-colorings of $G$:
\begin{eqnarray}\label{tapsi2}
&&M^n \cdot 2^{3M^2/2} \cdot 2^{(2\eta+H(\eta))n^2} \cdot 3^{\left(80 \eta+30 H(\eta)\right)n^2/2} \cdot 2^{(1-\xi) \binom{n}{2}}\nonumber\\
&\leq& M^n \cdot 2^{3M^2/2} \cdot  2^{\left(164 \eta+62 H(\eta)\right)n^2/2} \cdot 2^{(1-\xi) \binom{n}{2}} \leq   M^n \cdot 2^{3M^2/2} \cdot  2^{ \binom{n}{2} - \eta n^2} \leq \frac{1}{4} \cdot 2^{\binom{n}{2}}
\end{eqnarray}

Combining equations~\eqref{tapsi1} and~\eqref{tapsi2}, we derive $c_{3,\widehat{F}_3}(G) < 2^{\binom{n}{2}}$, which proves part (b). 
\end{proof}

We conclude this section with the following conjecture.

\begin{conjecture}
The only extremal graph for $c_{3,\widehat{F}_3}(G)$ is the complete graph $K_n$. 
\end{conjecture}

Comparing the number of good colorings of an almost complete graph with the number of good colorings of the complete graph seems to be hard. We did not find, for example, a way to construct an injection from the colorings of $K_n-e$ (where $e$ is any edge) to those of $K_n$. On the other hand, finding better upper bounds for $c_{3,\widehat{F}_3}(K_n)$ could be a first step to later improve the general bound on Theorem \ref{thm:approx_rainbow}. In the light of this, we state the following theorem, which has a simple proof.

\begin{theorem}\label{thm:ncoloringsKn}
The number of $3$-edge colorings of $K_n$ avoiding rainbow triangles satisfies $$c_{3,\widehat{F}_3}(K_n) \le \frac{3}{2}(n-1)!\cdot 2^{\binom{n-1}{2}}.$$
\end{theorem}

The above theorem is an easy consequence of the following lemma.

\begin{lemma} \label{lemma:extendingKn} Let $t\ge 2$ and consider the complete graph $K_{t+1}$ on vertices $v_1, \ldots, v_{t+1}$. Let $\hat{K}_t$ be any $3$-coloring of the edges induced by $v_1, \ldots, v_{t}$ which avoids a rainbow triangle. Then the number of ways to color the edges incident to $v_{t+1}$, still avoiding a rainbow triangle, is at most $t 2^t$. In other words, $c_{3,\widehat{F}_3}(v_{t+1}, \hat{K}_t) \le t 2^t$.
\end{lemma}
\begin{proof} We prove this by induction on $t$. For $t=2$, it is easy to check that we have $c_{3,\widehat{F}_3}(v_{3}, \hat{K}_2) = 7 < 2\cdot 2^2$. Assume that $t > 2$ and that the claimed result holds for smaller complete graphs. Let $v = v_{t+1}$ and fix a coloring of $\hat{K}_t$ as in the statement of this lemma. Let $u$ be any vertex of $K_t$. Let $N^1$, $N^2$, and $N^3$ be the set of vertices in $\hat{K}_t-u$ which are adjacent to $u$ by an edge of color $1$, $2$, and $3$, respectively. Finally, let $n_i = |N^i|$, so that $n_1+n_2+n_3 = t-1$. We count the number of ways to color the edges from $v$ to $\hat{K}_t$ for each fixed color of the edge $vu$. First assume that $vu$ receives color $1$. Then all edges from $v$ to $N_2$ cannot receive color $3$, and all edges from $v$ to $N_3$ cannot receive color $2$. Therefore, there are $2^{n_2+n_3}$ ways to color the edges from $v$ to $N_2\cup N_3$. Define $f:\mathbb{N} \to \mathbb{N}$ such that $f(0) = f(1) = 1$ and $f(n) = 0$ for $n\ge 2$. We argue that the number of ways to color the edges from $v$ to $N_1$ is at most $n_12^{n_1}+f(n_1)$. In fact, this is trivial to check for $n_1 = 0$ and $n_1=1$. Finally, since $n_1 \le t-1$, for $n_1 \ge 2$ we can use induction: so we can color the edges from $v$ to $N_1$ in at most $n_1 2^{n_1} = n_1 2^{n_1} + f(n_1)$ ways. This gives a total of $(n_12^{n_1} +f(n_1))2^{n_2+n_3}$ ways to color the edges from $v$ to $N_1\cup N_2\cup N_3$ (given that $uv$ is of color 1). The cases in which $vu$ is of color $2$ or $3$ are analogous. Adding the values in each case and using that $f(n_i) \le 1$ and $t>2$, gives us
\begin{align*}
c_{3,\widehat{F}_3}(v, \hat{K}_t) &\le (n_1+n_2+n_3) 2^{n_1+n_2+n_3} + f(n_1) 2^{n_2+n_3} + f(n_2)2^{n_1+n_3}+f(n_3) 2^{n_1+n_2} \\
						&\le (t-1)2^{t-1} + 2^{t-1}+ 2^{t-1}+ 2^{t-1} \\
						&= (t+2)2^{t-1} \\
						&\le t2^t.
\end{align*}
\end{proof}

\begin{proof}[Proof of Theorem \ref{thm:ncoloringsKn}]Let $v_1, v_2, \ldots, v_n$ be any ordering of the vertices of $K_n$. Applying Lemma \ref{lemma:extendingKn} for $t \in \{2, \ldots, n-1\}$, we obtain
$$c_{3,\widehat{F}_3}(K_n) \le c_{3,\widehat{F}_3}(K_2)\left(\prod_{t=2}^{n-1}t2^t \right) = 3(n-1)!\cdot 2^{\binom{n-1}{2}-1}.$$
\end{proof}


\section{$3$-colorings avoiding patterns with two colors}
\label{sec:twocolors}

Given a graph $F$, the \emph{Ramsey number} $R(F,F)$ is the smallest number $\ell$ such that any edge-coloring of the complete graph $K_\ell$ with two colors contains a monochromatic copy of $F$. In this section, we prove Theorem~\ref{thm_patterns}, which we restate below.

\begingroup
\def\thetheorem{\ref{thm_patterns}}
\begin{theorem}
\theoremTwocolors
\end{theorem}
\addtocounter{theorem}{-1}
\endgroup

In particular, this works for patterns for which one of the classes contains all but at most $\lceil (k-1)/2 \rceil$ edges incident with one of its vertices. 

Our strategy to prove Theorem~\ref{thm_patterns} is to adapt the general steps of the proof of Theorem 1.1 in Alon, Balogh, Keevash and Sudakov~\cite{ABKS} (see also Theorem~1 in~\cite{balogh06}) to our context. This involves proving a stability result, which shows that any graph $G$ with a large number of feasible colorings is similar to $T_{k-1}(n)$, and then proving the desired result by contradiction: starting with a counterexample on $n$ vertices, one shows that it is possible to find a counterexample on $n-1$ vertices whose `gap' to the desired optimal solution increases. A recursive application of this step would lead to an $\sqrt{n}$-vertex graph whose number of $3$-edge colorings that avoid $\widehat{F}$ is too high to be feasible.

To implement this idea, we define our concept of stability.
\begin{definition}\label{colored_stability}
A pattern $\widehat{F}$ of $K_k$ which has at most 3 classes is said to satisfy the $3$-stability Property if, for every $\delta>0$, there exists $n_0$ as follows. If $n>n_0$ and $G$ is an $n$-vertex graph such that $c_{3,\widehat{F}}(G) \geq 3^{\ex(n,F)}$, then there exists a partition $V(G)=V_1 \cup \cdots \cup V_{k-1}$ such that $\sum_{i=1}^{k-1} e(V_i) \leq \delta n^2$.
\end{definition}

Note that if a pattern $\widehat{F}$ satisfies the $3$-stability Property,
then for any $\delta > 0$ and $n$ sufficiently large, it follows immediately that $c_{3,\widehat{F}}(n) \leq 3^{\ex(n,F)+ \delta n^2}$. To prove Theorem~\ref{thm_patterns}, we shall demonstrate two auxiliary lemmas. The first states that any pattern $\widehat{F}$ as in the statement of the theorem satisfies the $3$-stability Property, while the second states that the Tur\'{a}n graph $T_{k-1}(n)$ is the unique extremal graph for patterns of $K_k$ that satisfy the $3$-stability Property. 

\begin{lemma}\label{lemma_patterns}
Let $k \geq 3$ and let $\widehat{F}$ be a pattern of $K_k$ with two classes, one of which induces a graph $J$ such that $R(J,J) \leq k$. Then $\widehat{F}$ satisfies the $3$-stability Property.
\end{lemma}

\begin{lemma}\label{lemma_final}
Let $k \geq 3$ and let $\widehat{F}$ be a pattern of $K_k$ that satisfies the $3$-stability Property (in particular, it has at most 3 classes). Then the equality $c_{3,\widehat{F}}(G)=c_{3,\widehat{F}}(n)$ is achieved by an $n$-vertex graph $G$ if and only if $G$ is isomorphic to the Tur\'{a}n graph $T_{k-1}(n)$.
\end{lemma}

As we have seen, the rainbow pattern of $K_3$ is a pattern that does not satisfy the $3$-stability Property, so that Lemma~\ref{lemma_final} does not apply in this case. 
 
 \subsection{Proof of Lemma~\ref{lemma_patterns}}

In this section, we shall prove Lemma~\ref{lemma_patterns}. To this end, let $k \geq 3$ and consider a pattern $\widehat{F}$ of $K_k$ as in the statement of the lemma. Fix $\delta>0$, which we may assume to satisfy $\delta<1$. Let $\beta>0$ and $m_1$ be given by Theorem~\ref{simonovits_stability} with $\alpha=\delta^2/(2^{10}(k-1)^4)$. 

With foresight, consider a parameter $\eta>0$ satisfying the following inequality:
\begin{eqnarray}
&&22H(\eta)+44 \eta <\min\left\{\beta,\frac{\delta^2}{2^{10}(k-1)^4}\right\} \label{eta1}
\end{eqnarray}
Note that $\alpha$ and $\eta$ are bounded above by $\delta/16$. Let $\eps>0$ and $n_1$ be given by Lemma~\ref{lemma_embedding}, and assume that $\eps<\eta/4$. Consider $n_2$ and $M$ given by Lemma~\ref{colored_regularity} with $m=\max\{1/\eps,m_1\}$.

Let $n_0 \geq \max\{n_1,n_2\}$ such that~\eqref{UB3} is satisfied for all $n \geq n_0$. For $n \geq n_0$, let $G=(V,E)$ be an $n$-vertex graph such that $c_{3,\widehat{F}}(G) \geq 3^{\ex(n,F)}$ and fix an arbitrary 3-edge-coloring of $G$ that avoids $\widehat{F}$. Consider the partition $V_1 \cup \cdots \cup V_t$ associated with this coloring given by Lemma~\ref{colored_regularity} with $m=1/\eps$. Let $R_1$, $R_2$ and $R_3$ be the cluster graphs (with minimum density $\eta/4$) associated with each of the three colors, and let $R$ be the corresponding multicolored cluster graph.

Exactly as in Theorem~\ref{thm:approx_rainbow}, defining $E_s$ as the set of edges that appear in exactly $s$ of the cluster graphs and denote $e_s=|E_s|$, we bound the number of $3$-edge colorings of $G$ that could give rise to this particular partition and cluster graphs (see equation \eqref{eq:numberOfClusterGraphs} and the observation after it): 

\begin{eqnarray}
&&\binom{n^2}{\eta n^2} 3^{\eta n^2} (1^{e_1}2^{e_2}3^{e_3})^{n^2/t^2}.\label{UB1}
\end{eqnarray}

To find an upper bound on $(1^{e_1}2^{e_2}3^{e_3})^{n^2/t^2}$, we define $R_j'=R_j-E_1$, so that  $2e_2+3e_3=e(R'_1)+e(R'_2)+e(R'_3)$. Suppose for a contradiction that $e(R_j') > t_{k-1}(t)$, so that there is a monochromatic copy of $K_k$. For the sake of the argument, assume that it is green. Because the edges of this copy of $K_k$ are not in $E_1$, it is possible to assign color red or blue to each edge $e$ so that the edge $e$ is an edge in the corresponding cluster graph $R_j'$. Since $R(J,J)\leq k$, there is a monochromatic copy of $J$ in blue or red in this copy of $K_k$. Combined with green edges, we generate a copy of $K_k$ with the forbidden pattern, which, because of Lemma~\ref{lemma_embedding}, contradicts the fact that the original coloring did not contain a copy of $\widehat{F}$.

Hence $2e_2+3e_3 \leq 3\ex(n,K_k)$ and we obtain $\frac{e_2}{t^2} \leq \frac{3(k-2)}{4(k-1)}-\frac{3e_3}{2t^2}$. Since $2<3^{7/11}$, and using the bound in~\eqref{entropy},the upper bound~\eqref{UB1} becomes at most
\begin{eqnarray*}
&&2^{H(\eta)n^2} 3^{\eta n^2} (1^{e_1}2^{e_2}3^{e_3})^{n^2/t^2} \nonumber\\
&<&3^{(H(\eta)+\eta)n^2+\left(\frac{21(k-2)}{22(k-1)}+\frac{e_3}{11 t^2}\right)\frac{n^2}{2}}. 
\end{eqnarray*}
We have
\begin{eqnarray}\label{UB2}
c_{3,\widehat{F}}(G) &\leq& \sum_{R} 3^{(H(\eta)+\eta)n^2+\left(\frac{21(k-2)}{22(k-1)}+\frac{e_3}{11 t^2}\right)\frac{n^2}{2}},
\end{eqnarray}
where the sum is over multicolored cluster graphs $R$ defined by triples $(R_1,R_2,R_3)$.

First assume that $e_3 < \left(\frac{k-2}{k-1}-88 \eta -44H(\eta)\right)\frac{t^2}{2}$ for \emph{all} such $R$. The number of vertex partitions is clearly bounded above by $M^n$, while the number of possible choices for $R_1$, $R_2$ and $R_3$ is at most $2^{3t^2/2} \leq 2^{3M^2/2}$. Equation~\eqref{UB2} leads to
\begin{eqnarray}\label{UB3}
c_{3,\widehat{F}}(G) &\leq& M^n \cdot 2^{3M^2/2} \cdot 3^{\left((k-2)/2(k-1) - \eta \right)n^2} < 3^{t_{k-1}(n)},
\end{eqnarray}
for $n$ sufficiently large, a contradiction.
 
In particular there must be a multicolored cluster graph $R=(R_1,R_2,R_3)$ for which $e_3 \geq \left(\frac{k-2}{k-1}-88 \eta -44H(\eta)\right)\frac{t^2}{2}$, where $V_1 \cup \cdots \cup V_t$ is the corresponding $\eps$-regular partition. By our choice of $\beta$ and $m$, Theorem~\ref{simonovits_stability} ensures that there is a partition $V(R)=W_1 \cup \cdots \cup W_{k-1}$ of the vertex set of $R'$ with $\sum e(W_i)<\alpha t^2$. Let $B$ be the $(k-1)$-partite subgraph of $R$ induced by the classes $W_1,\ldots,W_{k-1}$, so that 
\begin{eqnarray*}
e(B) &\geq& \left(\frac{k-2}{k-1}-88 \eta -44H(\eta)-2\alpha\right)\frac{t^2}{2}.
\end{eqnarray*}

The following lemma implies that the size of each $W_i$ is not far from $t/(k-1)$. A proof of this fact may be found in~\cite{HLO15}.

 \begin{lemma} \label{prop:prop1}
Let $H=(W,E)$ be a $(k-1)$-partite graph on $t$ vertices with $k$-partition $W= W_1 \cup \cdots 
\cup W_{k-1}$. If, for some $m \geq (k-1)^2$, the graph $H$ contains at least $\ex (t, K_{k}) - m$
edges, then for $i \in \{1, \ldots , k-1\}$ we have
$$
\left||W_i| - \frac{t}{k-1} \right| \leq \sqrt{\frac{2(k-2)}{k-1} \cdot m + 2(k-2)} < \sqrt{2m}.
$$
\end{lemma}
Applying this to the graph $B$, we deduce that $|W_i-t/(k-1)| \leq \gamma t$ for all $i \in \{1,\ldots,t\}$, where $\gamma= \sqrt{88 \eta + 44H(\eta)+2 \alpha}$. Our choice of $\eta$ and $\alpha$ implies that $\gamma \leq \frac{\delta}{16(k-1)^2}$.
Clearly, each edge removed from the Tur\'{a}n graph $T_{k-1}(t)$ to produce $B$ eliminates at most $(t/(k-1)+\gamma t)^{k-3}$ copies of $K_{k-1}$, so that $B$ contains at least
\begin{eqnarray}\label{eq_B}
&&\left( \frac{t}{k-1} \right)^{k-1}-(44 \eta +22H( \eta)+\alpha)t^2 \cdot \left( \frac{t}{k-1}+\gamma t\right)^{k-3}\nonumber\\
&\geq&  \frac{t^{k-1}}{(k-1)^{k-3}} \cdot \left((k-1)^2-\gamma \right).
\end{eqnarray} 
copies of $K_{k-1}$. For the last inequality, we used that $44 \eta +22H( \eta)+\alpha < \gamma^2$ and, since $\gamma(k-1)<1$, we have
\begin{eqnarray*}
&&\left( \frac{t}{k-1}+\gamma t\right)^{k-3}=t^{k-3} \gamma^{k-3} \sum_{i=0}^{k-3} \left(\frac{1}{\gamma (k-1)}\right)^i \leq t^{k-3} \gamma^{k-3} \frac{1}{(k-1)^{k-2}\gamma^{k-2}}\\ 
&=& \left(\frac{t}{k-1}\right)^{k-3} \frac{1}{(k-1) \gamma} < \left(\frac{t}{k-1}\right)^{k-3} \frac{1}{\gamma}.
\end{eqnarray*}
Let $U_1 \cup \cdots \cup U_{k-1}$ be the partition of $V(G)$ given by $U_i=\cup_{j \in W_i} V_j$. We argue that the number of edges in $\cup_{i=1}^{k-1} G[U_i]$ is small. First note that:
\begin{itemize}
\item[(i)] the number of edges that come from a pair of classes $(V_j,V_{j'})$ such that $\{j,j'\} \in E(R')=E_3$ is at most $\alpha t^2 \cdot (n^2/t^2)=\alpha n^2$;

\item[(ii)] the number of edges that come from a pair of classes $(V_j,V_{j'})$ such that $\{j,j'\} \notin E_1 \cup E_2 \cup E_3$ because the pair is not $\eps$-regular for at least one of the colors is at most $3 \eps t^2 (n/t)^2=3 \eps n^2$;

\item[(iii)] the number of edges that come from a pair of classes $(V_j,V_{j'})$ such that $\{j,j'\} \notin E_1 \cup E_2 \cup E_3$ because the pair is sparse for all colors is at most $3 \eta n^2$;

\item[(iv)] the number of edges with both endpoints in a same set $V_j$ is bounded above by $t (n/t)^2=n^2/t\leq \eps n^2$.
\end{itemize}

It remains to bound the number of edges in pairs $(V_j,V_{j'})$ such that $j,j' \in W_i$ for some $i \in \{1,\ldots,k-1\}$, with the additional properties that  $\{j,j'\} \in E_1 \cup E_2$ and $(V_j,V_{j'})$ is $\eps$-regular for all colors. Let  $(V_j,V_{j'})$ and assume that $i=k-1$. 

\begin{claim}
There are no sets $V_{j_1},\ldots,V_{j_{k-2}}$, where $j_\ell \in W_\ell$ for all $\ell \in \{1,\ldots,k-2\}$, such that both $\{j_1,\ldots,j_{k-2},j\}$ and $\{j_1,\ldots,j_{k-2},j'\}$ induce copies of $K_{k-1}$ in $R'$.
\end{claim}

\begin{proof} 
Assume for a contradiction that there are such sets and let $\sigma$ be a color for which $\{j,j'\} \in E(R_\sigma)$. This implies that $\{j_1,\ldots,j_{k-2},j,j'\}$ induces a copy of $K_k$ in $R_\sigma$ and a copy of $K_k-\{j,j'\}$ in the cluster graphs corresponding to the other colors. This clearly leads to a copy of $K_k$ colored according to $\widehat{F}$ (where $\sigma$ is one of the colors) in the multicolored cluster graph. Lemma~\ref{lemma_embedding} leads to the desired contradiction. 
\end{proof}

To conclude the proof, we find an upper bound on the number $N$ of pairs $j,j'$ in a same set $W_i$ for which there are no sets $V_{j_1},\ldots,V_{j_{k-2}}$, one in each of the remaining classes $W_{\ell}$, such that both $\{j_1 \ldots,j_{k-2},j\}$ and $\{j_1,\ldots,j_{k-2},j'\}$ induce copies of $K_{k-1}$ in $R$. 

Clearly, any vertex $s \in W_i$ could lie in at most $(t/(k-1)+\gamma t)^{k-2}$ copies of $K_{k-1}$ with one vertex in each set $W_{\ell}$, so that, to avoid the occurrence of the above sets $V_{j_1},\ldots,V_{j_{k-2}}$, at least one of $j,j'$ lies in at most 
$$m=\frac{1}{2} \left(\frac{t}{k-1}+\gamma t\right)^{k-2}$$
copies of $K_{k-1}$ in $R$ by the pigeonhole principle. Let $A$ be the number of elements $s \in [t]$ which lie in at most $m$ such copies of $K_{k-1}$. Clearly $N \leq (k-1) \cdot \left(\frac{t}{k-1} + \gamma t \right) \cdot A \leq  2 \cdot t \cdot A$. To find an upper bound on $A$, consider the auxiliary bipartite graph $B'$ whose bipartition is given by $[t]=V(B)$ and by the set $K^B_{k-1}$ of copies of $K_{k-1}$ in $B$. We add an edge $\{s,K\}$ whenever vertex $s$ lies in the clique $K$. Clearly, $A$ is the number of elements $s \in [t]$ with degree at most $m$ in $B'$. The number of edges in $B'$ is $(k-1)\left| K^B_{k-1} \right|$. By~\eqref{eq_B} we have 
$$\frac{A}{2} \left(\frac{t}{k-1}+\gamma t\right)^{k-2} + (t-A)  \left(\frac{t}{k-1}+\gamma t\right)^{k-2} \geq e(B') \geq \frac{t^{k-1}}{(k-1)^{k-4}} \cdot \left((k-1)^2-\gamma \right).$$
This leads to
$$A \leq 2(k-1) \gamma t+2\gamma(k-1)^2t<4\gamma(k-1)^2t.$$
In particular, the number of edges in $G$ with endpoints in sets $V_j,V_{j'}$ that are contained in the same $W_i$ with the additional property that $j$ or $j'$ lie in at most $m$ copies of $K_{k-1}$ in $R$ with one vertex in each set $W_{\ell}$ is bounded above by 
$$N \left(\frac{n}{t}\right)^2 \leq 8 \gamma (k-1)^2 t^2 (n/t)^2 = 8 \gamma (k-1)^2 n^2.$$ 

Putting everything together, we conclude that the number of edges in $G$ with endpoints in a same $U_i$ is at most
\begin{eqnarray*}
&&\alpha n^2+3 \eps n^2+3\eta n^2+\eps n^2+8 \gamma (k-1)^2 n^2\\
&\leq& \alpha n^2 + 4 \eta n^2 + \delta n^2/2,
\end{eqnarray*}
which is smaller than $\delta n^2$ by the definition of $\gamma$ and our choice of $\eta$, $\eps$ and $\alpha$.

\subsection{Proof of Lemma~\ref{lemma_final}}

The aim of this section is to prove Lemma~\ref{lemma_final}, which states that the Tur\'{a}n graph $T_{k-1}(n)$ is extremal for any pattern $\widehat{F}$ of $K_k$ that satisfies the $3$-stability Property. We shall use the following result from~\cite{ABKS}.
\begin{lemma}\cite{ABKS}\label{lemma_alon}
Let $G$ be a graph and $W_1,\ldots,W_k$ be subsets of vertices of $G$ such that, for every pair $i \neq j$ and every pair of subsets $X_i \subset W_i$ and $X_j \subset W_j$ with $|X_i| \geq 10^{-k} |W_i|$ and $|X_j| \geq 10^{-k} |W_j|$, there are at least $|X_i||X_j|/10$ edges between $X_i$ and $X_j$ in $G$. Then $G$ contains a copy of $K_k$ with one vertex in each set $W_i$.
\end{lemma}

Now, we can give a proof for Lemma~\ref{lemma_final}.

\begin{proof}[Proof of Lemma~\ref{lemma_final}] Our proof is inspired by the proof of Theorem 1.1 in~\cite{ABKS}, which may be slightly shortened because of Lemma~\ref{lemma:containsmultipartite}. Let $k \geq 3$ and let $\widehat{F}$ be a pattern of $K_k$ satisfying the $3$-stability property (see Definition~\ref{colored_stability}). For fixed $\delta>0$, which will be chosen conveniently later, let $n_0$ be given as in the definition of $3$-stability. 

Suppose that $G=(V,E)$ is a graph on $n > n_0^2$ vertices with at least $3^{t_{k-1}(n)+m}$ distinct $3$-edge colorings that avoid $\widehat{F}$, for some $m \geq 0$. We claim that, if we assume that $m>0$, then $G$ contains a vertex $v$ such that $G-v$ has at least $3^{t_{k-1}(n-1)+m+1}$ distinct $3$-edge colorings that avoid $\widehat{F}$. Repeating this argument, we obtain a graph on $n_0$ vertices with at least $3^{t_{k-1}(n_0)+m+n-n_0}>3^{n_0^2}$ such 3-colorings. This is a contradiction, as a graph on $n_0$ vertices has at most $n_0^2/2$ edges, and hence at most $3^{n_0^2/2}$ distinct $3$-edge colorings.

To implement this idea, suppose that $m>0$. The above claim holds easily if $\delta(G)<\delta_{k-1}(n)$, where $\delta_{k-1}(n)$ denotes the minimum degree of $T_{k-1}(n)$, as in this case we would be able to delete a vertex $v$ of minimum degree and the number of colorings of $G-v$ that avoid $\widehat{F}$ would be at least
$$3^{-\delta(G)}3^{t_{k-1}(n)+m} \geq 3^{t_{k-1}(n)-\delta_{k-1}(n)+m+1}= 3^{t_{k-1}(n-1)+m+1}.$$

Thus we assume that $\delta(G) \geq \delta_{k-1}(n)$. Let $V_1 \cup \cdots \cup V_{k-1}$ be a partition of $V$ that minimizes $\sum_i e(V_i)$, so that it satisfies $\sum_i e(V_i)<\delta n^2$ by the choice of $\delta$ and $n_0$ in terms of the $3$-stability property. If we fix $\delta=10^{-11k}$, because of Lemma~\ref{prop:prop1} we can easily claim that 
$||V_i|-n/(k-1)| < \sqrt{2 \cdot 10^{-11k} n^2} < \sqrt{2/10^k \cdot 10^{-10k} n^2} < \sqrt{2/10^k} \cdot 10^{-5k}n < 10^{-5k}n$.

First assume that $G$ contains a vertex $v$ with at least $n/(10^3k)$ neighbors within its own class. The minimality of $\sum_i e(V_i)$ implies that $v$ has at least this many neighbors in each of the classes $V_j$ (otherwise we would move $v$ to a different class). Given a feasible coloring of $G$, we say that a color $\sigma$ is \emph{rare} with respect to $v$ and $V_i$ if it appears at most $n/(10^{3+k})$ times in edges between $v$ and $V_i$, otherwise it is called \emph{abundant}.  A class $V_i$ is said to be \emph{$s$-weak} if there are $s$ rare colors with respect to $v$ and $V_i$. More generally, a class is said to be \emph{weak} if it is either $1$- or $2$-weak. Because $v$ has a large number of neighbors in each class, note that, for every $i$, at least one of the colors is abundant with respect to $v$ and $V_i$, and hence no class is $3$-weak (remember that we only use tree colors). We split the set $\mathcal{C}$ of 3-colorings of $G$ that avoid $\widehat{F}$ into classes $\mathcal{C}_1 \cup \mathcal{C}_2$, where $\mathcal{C}_1$ contains colorings that we now describe. There is a choice of colors $\sigma_1,\ldots,\sigma_{k-1}$ satisfying the following two properties: (i) the number of occurrences of each color coincides with the distribution of colors in the neighborhood of some vertex $x$ of $K_{k}$ in some coloring according to pattern $\widehat{F}$; (ii) for every $i \in \{1,\ldots,k-1\}$, color $\sigma_i$ is abundant with respect to $v$ and $V_i$. In other words, a coloring lies in $\mathcal{C}_1$ if it allows a partial embedding of $\widehat{F}$ into $G$ where $x$ is mapped to $v$ is one of the vertices, the neighbors of $x$ are mapped to distinct classes $V_i$, and the colors between $x$ and its neighbors are all abundant with respect to $v$ and the respective $V_i$.

We observe that, if a coloring lies in $\mathcal{C}_2$, either there are three or more weak classes or there are exactly two weak classes and at least one of the classes is $2$-weak. Indeed, if there are only two weak classes $V_{i}$ and $V_j$ and both are $1$-weak, we may clearly choose abundant colors $\sigma_i$ and $\sigma_j$ with respect to $V_i$ and $V_j$, respectively, regardless of whether we want them to be different or the same. Since all colors are abundant for any remaining class, we can always extend this to a partial embedding that respects the pattern $\widehat{F}$. Observe, however, that this is not necessarily the case when three classes are weak, as we can avoid monochromatic neighborhoods (by assigning distinct sets of colors to the weak classes) or a pattern of $K_4$ where each edge is incident with three colors (by assigning the same set of colors to all weak classes). We may also avoid monochromatic patterns if there are two weak classes, and one of them is $2$-weak (as we may assign disjoint sets of colors to the weak classes).

We may deal with colorings in $\mathcal{C}_1$ as follows: let $\Delta$ be such a coloring, and let $\sigma_i$ and $W_i \subset V_i \cap N(v)$ be the colors and sets satisfying the definition, for $i \in \{1,\ldots,k-1\}$. If, given arbitrary sets $X_j \subset  W_j$ and $X_{j'}\subset  W_{j'}$ with $|X_{j}| \geq 10^{-{k+1}}|W_{j}|$ and $|X_{j'}| \geq 10^{-{k+1}}|W_{j'}|$, the number of edges of each color between $X_j$ and $X_{j'}$ is at least $|X_j| |X_{j'}|/10$ for all distinct $j,j'$, then we would find a forbidden copy of $K_k$ colored according to $
\widehat{F}$ by Lemma~\ref{lemma_alon}, a contradiction. Thus there are distinct $j,j' \in [k-1]$ and two reasonably large sets $X_j \subset V_j$ and $X_{j'} \subset V_{j'}$ (size at least $n/(10^{2k+2})$ by the definition of abundant color) such that are at most $|X_j| |X_{j'}|/10$ edges of some color $\sigma$ joining $X_j$ and $X_{j'}$. 

Because of this, to obtain an upper bound on $|\mathcal{C}_1|$, we may proceed as follows. There are at most $2^{2n}$ ways to choose $X_j$ and $X_{j'}$ (this is a rough upper bound that uses the fact that the vertex set of the graph has $2^n$ possible subsets.) Moreover, once these sets are chosen, the edges between them may be colored in at most
\begin{eqnarray}\label{boundXj}
&&\binom{|X_j| |X_{j'}|}{|X_j| |X_{j'}|/10} 2^{|X_j| |X_{j'}|} 3^{|E(G)|-|E(X_j,X_{j'})|}\\
&& \leq 2^{H(0.1) |X_j| |X_{j'}|}  2^{|X_j| |X_{j'}|} 3^{|E(G)|-|E(X_j,X_{j'})|} \leq 2^{3/2 |X_j| |X_{j'}|} 3^{|E(G)|-|E(X_j,X_{j'})|}\nonumber
\end{eqnarray} 
ways. By the fact that there are at most $t_{k-1}(n)+10^{-11k}n^2-|X_j| |X_{j'}|$ other edges in the graph, the number of colorings in $\mathcal{C}_1$ is bounded above by
\begin{eqnarray*}
&& 2^{2n}  \left(\frac{\sqrt{8}}{3}\right)^{ |X_j| |X_{j'}|} 3^{t_{k-1}(n)+10^{-11k}n^2} \\
&&\leq
\left( 2^{2n}  \left(\frac{\sqrt{8}}{3}\right)^{n^2/(10^{4k+4})} 3^{10^{-11k}n^2} \right)  3^{t_{k-1}(n)} \leq \left( 2^{2n}  3^{(10^{-11k}-10^{-4k-6})n^2} \right)  3^{t_{k-1}(n)}, 
\end{eqnarray*}
which is much smaller than $3^{t_{k-1}(n)}$ for sufficiently large $n$ because $\sqrt{8}/3 \leq 3^{-0.01}$. 

Since $|\mathcal{C}| \geq 3^{t_{k-1}(n)+m}$ by hypothesis, this bound on $|\mathcal{C}_1|$ implies that $|\mathcal{C}_2| \geq 3^{t_{k-1}(n)+m-1}$. By our previous discussion, there are two possibilities. Firstly, there may be three weak classes $V_{j_1}$, $V_{j_2}$ and $V_{j_3}$ (this may only happen for $k\geq 4$). Secondly, there may be two weak classes $V_{j_1}$ and $V_{j_2}$, where one of them, say $V_{j_1}$, is $2$-weak. 

Suppose that we are in the first case. The number of ways of choosing the classes $V_{j_1}$, $V_{j_2}$ and $V_{j_3}$ and coloring the edges between $v$ and $V_{j_1} \cup V_{j_2} \cup V_{j_3}$ is bounded above by
\begin{eqnarray*}
&&(3k)^3 \binom{|V_{j_1}|}{n/10^{3k}} \binom{|V_{j_2}|}{n/10^{3k}} \binom{|V_{j_3}|}{n/10^{3k}}2^{|V_{j_1}|+|V_{j_2}|+|V_{j_3}|} \\
&&\leq (3k)^3  \binom{(1/(k-1)+10^{-5k})n}{n/10^{3k}}^32^{(3/(k-1)+3\cdot 10^{-5k})n}\\
&& \leq 2^{(3 \cdot H(0.001)/(k-1)+3/(k-1)+3\cdot 10^{-5k})n} \leq  2^{(3.06/(k-1)+3 \cdot 10^{-5k})n}, 
\end{eqnarray*}
for large $n$, since $H(0.001)<0.02$. Moreover, $v$ is adjacent with at most $((k-4)/(k-1)+3 \cdot 10^{-5k})n$ vertices outside $V_{j_1} \cup V_{j_2} \cup V_{j_3}$, and hence the edges between $v$ and the remainder of the graph may be colored in at most $3^{((k-4)/(k-1)+3\cdot 10^{-5k})n}$ ways. 
Therefore the number of ways of coloring the edges incident with $v$ is at most,
\begin{eqnarray*}
&& 2^{(3.06/(k-1)+3 \cdot 10^{-5k})n} 3^{((k-4)/(k-1)+3\cdot 10^{-5k})n}.
\end{eqnarray*}

In the second case, we proceed similarly. The sets $V_{j_1}$ and $V_{j_2}$ may be chosen in at most $k^2$ ways, and the edges between $v$ and $V_1 \cup V_2$ may be colored in at most
\begin{eqnarray*}
&&\binom{|V_{j_1}|}{n/10^{3k}}^2 \binom{|V_{j_2}|}{n/10^{3k}} 2^{|V_{j_2}|} \leq \binom{(1/(k-1)+10^{-5k})n}{n/10^{3k}}^32^{(1/(k-1)+10^{-5k})n}\\
&& \leq 2^{(3 \cdot H(0.001)/(k-1)+1/(k-1)+ 10^{-5k})n} \leq  2^{(1.06/(k-1)+3 \cdot 10^{-5k})n}. 
\end{eqnarray*}
The remaining edges between $v$ and the other classes may be colored in at most $3^{((k-3)/(k-1)+2\cdot 10^{-5k})n}$ ways.

If $k \geq 4$, we have 
\begin{eqnarray*}
&&2^{(3.06/(k-1)+3 \cdot 10^{-5k})n} 3^{((k-4)/(k-1)+3\cdot 10^{-5k})n}+2^{(1.06/(k-1)+3 \cdot 10^{-5k})n}3^{((k-3)/(k-1)+2\cdot 10^{-5k})n}\\
&\leq & 2 \cdot 2^{(3.06/(k-1)+3 \cdot 10^{-5k})n} 3^{((k-4)/(k-1)+3\cdot 10^{-5k})n}, 
\end{eqnarray*}
and hence the number of colorings of $G-v$ is at least
\begin{eqnarray*}
&& 3^{t_{k-1}(n)+m-1}  \cdot 2^{-(3.06/(k-1)+4 \cdot 10^{-5k})n} 3^{-((k-4)/(k-1)+3\cdot 10^{-5k})n}\\
&\geq& 3^{t_{k-1}(n)-(k-4)/(k-1)n+m-1}  2^{-3.06n/(k-1)} 6^{-4 \cdot 10^{-5k}n} \\
&\geq& 3^{t_{k-1}(n-1)+m-1} 3^{2/(k-1)n}  2^{-3.06n/(k-1)} 6^{-4 \cdot 10^{-5k}n}\\
&\geq& 3^{t_{k-1}(n-1)+m-1} 3^{1.95/(k-1)n}  2^{-3.06n/(k-1)} \geq 3^{t_{k-1}(n-1)+m+1},
\end{eqnarray*}
because $2^{3.06}<3^{1.95}$, as required in this case.

If $k=3$, the number of colorings of $G-v$ is at least
\begin{eqnarray*}
&& 3^{t_{k-1}(n)+m-1}  \cdot 2^{-(1.06/(k-1)+3 \cdot 10^{-5k})n} 3^{-((k-3)/(k-1)+2\cdot 10^{-5k})n}\\
&\geq& 3^{t_{k-1}(n-1)+m-1} 3^{1/(k-1)n}  2^{-1.06n/(k-1)} 6^{-3 \cdot 10^{-5k}n}\\
&\geq& 3^{t_{k-1}(n-1)+m-1} 3^{0.95/(k-1)n}  2^{-1.06n/(k-1)} \geq 3^{t_{k-1}(n-1)+m+1},
\end{eqnarray*}
because $2^{1.06}<3^{0.95}$.

Finally, we consider the case where each vertex has fewer than $n/(10^3k)$ neighbors within their own class. Clearly, $n/(10^3k) < (|V_i|-2)/2$ for every $1\le i \le k-1$. Therefore, for each edge $vw$ with both ends inside some class $V_i$, there exists a vertex $u \in V_i$ which is not adjacent to both $v$ and $w$. By Lemma \ref{lemma:containsmultipartite} (applied to each such $vw$), we may delete every edge inside any part $V_i$, obtaining a graph $G^*$ which has at least as many coloring as $G$ and is $(k-1)$-partite graph. Therefore, $|E(G^*)| \le t_{k-1}(n)$. So the number of $3$-colorings avoiding $\widehat{F}$ of $G^*$, and so also of $G$, is at most $3^{t_{k-1}(n)}$. This contradicts the fact that $m > 0$.

\end{proof}

\section{Appendix}
Here we give a proof for Lemma~\ref{lemma:GraphMaxW}, which holds for graphs on any number of vertices.

\begin{proof}[Proof of Lemma~\ref{lemma:GraphMaxW}]
Fix some natural number $n$. Let $G$ be a graph on $n$ vertices which maximizes $w(G)$, where $w$ is given by Definition \ref{def:w}. For $v \in V(G)$ define $w(v)$ as the product of the weights of the edges incident with $v$. Clearly, $w(G) = \left(\prod_{v\in V(G)} w(v)\right)^{1/2}$. We will use an argument similar to the Zykov Symmetrization's proof of Tur\'{a}n's Theorem. 

First, we show that, for any two non-adjacent vertices $u, v$, we must have $w(u) = w(v)$. Suppose, for a contradiction, that there are non-adjacent $u, v \in V(G)$ such that $w(v) > w(u)$. We create a new graph $G^*$ from $G$ by deleting $u$ and cloning $v$, that is, adding a new vertex $v'$ adjacent to the same neighbors as $v$. Note that, when we delete $u$, the weight of the remaining edges may only increase. When we add $v'$, the weight of all edges in $G \setminus\{u\}$ will stay the same in $G^*$, since any edge of $G \setminus\{u\}$ belongs to a triangle in $G \setminus\{u\}$ if and only if it belongs to a triangle in $G^*$. Furthermore, in $G^*$ we have $w(v') = w(v)$. Therefore, we have that $w(G^*) > w(G)$, contradicting the fact that $w(G)$ is maximum.

If all the edges of $G$ have weight 2, then the result follows trivially. So, assume that $G$ has an edge of weight 3, and let $x$ be one of the endpoints of such an edge. We will prove that all vertices must have the same weight. Let $N(x)$ be the set vertices adjacent to $x$ and $\bar{N}(x)$ be set of vertices non-adjacent to $x$. Moreover, for $i = 2, 3$, let $N_i(x) = \{u \in V(G) : xu \in E(G) \text{ and } w(xu)=i\}$. Note that $N_3(x)$ is non-empty (while $\bar{N}(x)$ and $N_2(x)$ may be empty). Let $y \in N_3(x)$. Notice that, since edges of weight 3 do not belong to any triangle, there can be no edges inside $N_3(x)$ or from $N_3(x)$ to $N_2(x)$ (so vertices in $N_3(x)$ are isolated in $G[N(x)]$). By the previous discussion, we have that every vertex in $\bar{N}(x)$ must have weight equal to $w(x)$ and every vertex in $N(x)$ must have weight $w(y)$. Suppose, for a contradiction, that $w(x) \neq w(y)$. If there were two non-adjacent vertices $a, b$ with $a \in \{x\}\cup \bar{N}(x)$ and $b \in N(x)$, then we would have $w(a) = w(b)$ which implies that $w(x)= w(y)$. Therefore, the graph $G$ must contain all edges between $\{x\}\cup \bar{N}(x)$ and $N(x)$. We claim that, in this case, $\{x\}\cup \bar{N}(x)$ and $N(x)$ must be independent sets. To prove this, let $a \in \bar{N}(x)$. Notice that, since the weight of a vertex is a number of the form $2^p3^q$, by the unique factorization in primes, the fact that $w(a) = w(x)$ implies that $|N_i(a)| = |N_i(x)|$ for $i= 2, 3$. In particular, $|N(a)| = |N(x)|$. And since $a$ is adjacent to all elements in $N(x)$, it follows that $a$ cannot be adjacent to any element in $\{x\}\cup \bar{N}(x)$. Therefore, $\{x\}\cup \bar{N}(x)$ is independent. Similarly, since $y$ is isolated in $N(x)$ and adjacent to all vertices in $\{x\}\cup \bar{N}(x)$, and for any $b \in N(x)$ we have $w(b) = w(y)$, it follows that $b$ cannot have any neighbors in $N(x)$. Therefore, $N(x)$ is independent. It follows that $G$ is a complete bipartite graph. But this implies that $w(G) \le 3^{\ex(n, K_3)} < 2^{\binom{n}{2}}=w(K_n)$, which contradicts the fact that $w(G)$ is maximum.

Now we know that all vertices have the same weight $w(x)$. This implies that there are natural numbers $d_2, d_3, \bar{d}$ such that for all $v \in V(G)$, we have $|N_2(v)| = d_2$, $|N_3(v)| = d_3$ and $|\bar{N}(v)| = \bar{d}$. Now, as before, if it happens that $d_3 \neq 0$, we take $x$ and $y$ such that $y \in N_3(x)$ and note that all neighbors of $y$ belong to $\bar{N}(x)\cup\{x\}$. This implies that $d_2+d_3 \le \bar{d}+1$ and, since $d_2+d_3+\bar{d} = n-1$, we have $d_2+d_3 \le n/2$. Therefore, $w(G) =  \left(\prod_{v\in V(G)} w(v)\right)^{1/2} = \left((2^{d_2}3^{d_3})^n\right)^{1/2} \le 3^{n^2/4} < 2^{\binom{n}{2}}$.
\end{proof}


\end{document}